\documentclass[12pt]{article}
\usepackage{amsmath,amsthm,amssymb,color}
\input amssym.def
\input amssym.tex
\date{}

\newtheorem{lemma}{Lemma}[section]
\newtheorem{proposition}[lemma]{Proposition} 
\newtheorem{theorem}[lemma]{Theorem}
\newtheorem{corollary}[lemma]{Corollary} 
\newtheorem{remark}[lemma]{Remark}

\def\konst #1{{\color{blue} #1}}

\begin{document}

\title{{\bf New examples of $K$-monotone\\ weighted Banach couples}}

\author {Sergey V. Astashkin\thanks{Research partially supported by
RFBR grant 10-01-00077, G. S. Magnusons found of the Royal Swedish
Academy of Sciences-project number FOAMagn09-028 and Lule{\aa}
University of Technology.}, Lech Maligranda\thanks{Research partially supported by the Swedish
Research Council (VR) grant 621-2008-5058.}\\
 \,{\small and}
Konstantin E. Tikhomirov$^*$}

\date{}

\maketitle

\renewcommand{\thefootnote}{\fnsymbol{footnote}}


\vspace{-7 mm}

\begin{abstract}
\noindent {\footnotesize Some new examples of $K$-monotone couples of
the type $(X, X(w))$, where $X$ is a symmetric space on $[0, 1]$ and
$w$ is a weight on $[0, 1]$, are presented. Based on the property of
the $w$-decomposability of a symmetric space we show that, if a
weight $w$ changes sufficiently  fast, all symmetric spaces $X$ with
non-trivial Boyd indices such that the Banach couple $(X, X(w))$ is
$K$-monotone belong to the class of ultrasymmetric Orlicz spaces.
If, in addition, the fundamental function of $X$ is $t^{1/p}$ for
some $p \in [1, \infty]$, then $X = L_p$. At the same time a Banach
couple $(X, X(w))$ may be $K$-monotone for some non-trivial $w$ in
the case when $X$ is not ultrasymmetric. In each of the cases where $X$
 is a Lorentz, Marcinkiewicz or Orlicz space we have found conditions which
 guarantee that $(X, X(w))$ is $K$-monotone.}
\end{abstract}


\begin{section}
{\bf Introduction}
\end{section}

One of the fundamental problems in interpolation theory is to find a description of all interpolation 
spaces between two fixed Banach spaces $X_0$ and  $X_1$, which form a Banach couple ${\bar X} = (X_0, X_1)$, 
i.e., the description of all intermediate Banach spaces $X$ with respect to ${\bar X}$ such that every linear operator 
$T\colon {\bar X} \to {\bar X}$ maps $X$ into $X$ boundedly.

An important role in the interpolation theory is played by the {\it $K$-monotone spaces} between fixed Banach spaces 
$X_0$ and  $X_1$, which are defined as follows: if $x\in X$, $y\in X_0 + X_1$, and the inequality
$$
K(t,y; X_0, X_1) \leq K(t,x; X_0, X_1) ~~{\rm holds ~for ~all} ~ t > 0,
$$
then $y\in X$ and $\|y\|_{X} \leq C \|x\|_{X}$ for some constant $C \geq 1$ independent of $x$ and $y$. Here
$$
K(t,x; X_0, X_1)= \inf \{ \|x_0\|_{X_0} + t\|x_1\|_{X_1}: x =  x_0 +x_1, x_0 \in X_0, x_1 \in X_1\}
$$
is the classical $K$-functional of Peetre.

A couple ${\bar X} = (X_0, X_1)$ is called {\it $K$-monotone} (or {\it Calder\'on-Mityagin couple})  if all interpolation spaces 
between $X_0$ and $X_1$ are $K$-monotone.

By a theorem due to Brudny{\u \i} and Krugljak \cite[Theorem 4.4.5]{BK91} all interpolation spaces with respect to a $K$-monotone 
Banach couple $(X_0, X_1)$ can be represented in the form $X = (X_0, X_1)_{\Phi}^K$, where $\Phi$ is a Banach lattice of measurable 
functions on $(0, \infty)$ and
$$
\| x\|_{(X_0, X_1)_{\Phi}^K} = \| K(\cdot, x; X_0, X_1)\|_{\Phi}.
$$
Moreover, even if $(X_0, X_1)$ is not $K$-monotone, every interpolation space $X$ with respect to $(X_0, X_1)$ which happens 
to be a $K$-monotone space satisfies $X = (X_0, X_1)_{\Phi}^K$ for some suitable $\Phi$, and of course this 
is only up to equivalence of norms ( Brudny{\u \i} and Krugljak \cite[Theorem 3.3.20]{BK91}).
Therefore, the problem of finding new examples of $K$-monotone couples or $K$-monotone spaces becomes very important.

Calder\'on \cite{Ca66} and independently Mitjagin \cite{Mi65} proved that the couple $(L_1, L_{\infty})$ is $K$-monotone. 
Several years later Sedaev and Semenov \cite{SS71} proved that a weighted couple
$(L_1(w_0), L_1(w_1))$ is $K$-monotone (cf. also Cwikel-Kozlov
\cite{CK02} for another proof) and then Sedaev \cite{Se73}
generalized this result to the couples of the form $(L_p(w_0),
L_p(w_1))$ $(1 \leq p \leq \infty)$. Finally, Sparr \cite{Sp74},
\cite{Sp78} showed that $(L_p(w_0),L_q(w_1))$ is a $K$-monotone couple 
for $0 < p, q \leq \infty$. There are other proofs of Sparr's result,
 for example, in papers of Dmitriev \cite{Dm81}, Cwikel \cite{Cw76} and of Arazy-Cwikel \cite{AC84}.

In \cite{CN03}, Cwikel and Nilsson considered the problem of $K$-monotonicity from 
a somewhat different point of view. Namely, they studied the problem when
a weighted Banach couple $(X(w_0), Y(w_1))$, with $X, Y$ being separable
Banach lattices with the Fatou property on a measure space
$(\Omega, \Sigma, \mu)$, is $K$-monotone for all weights $w_0, w_1$
on $\Omega$. They proved that this can happen if and only if $X =
L_p(v_0)$ and $Y = L_q(v_1)$ for some weights $v_0,v_1$ and some
numbers $1 \leq p, q < \infty$. In their proof the concept of \konst{a}
decomposable Banach lattice on a measure space is essentially used.
A Banach lattice $X$ is called {\it decomposable} if for any convergent series
$\sum_{n=1}^\infty f_n$ in $X$ with pairwise disjoint $f_n$ $(n=1,2,\dots)$
and any (formal) series $\sum_{n=1}^\infty g_n, g_n\in X$, $\|g_n\|_X\leq\|f_n\|_X$
$(n=1,2,\dots)$, such that all $g_n$ are pairwise disjoint, we have
$\sum_{n=1}^\infty g_n\in X$ and 
$\left\|\sum_{n=1}^\infty g_n\right\|_X\leq C
\left\|\sum_{n=1}^\infty f_n\right\|_X$ with
a constant $C$ independent of $f_n$, $g_n$.
This notion or some variants of it were introduced earlier by Cwikel
\cite{Cw84} and Cwikel-Nilsson \cite{CN84}.

\smallskip
Note that the problem of $K$-monotonicity of weighted couples
$(X(w_0), Y(w_1))$ can be reduced to considering couples of the
form $(X, Y(w))$. Therefore, in what follows, we will examine
couples with one weight only. We will say that a weight $w$ is non-trivial if 
either $w$ or $1/w$ is unbounded.

In \cite{Ti11}, the concept of $w$-decomposability of a Banach lattice, which 
generalizes in a sense the previous one due to Cwikel, was introduced. 
A theorem proved in \cite{Ti11} states that, whenever $X$ is a Banach lattice with 
the Fatou property, the couple $(X, X(w))$ is $K$-monotone if and only if $X$ is
$w$-decomposable (see Theorem \ref{Theor Tikhomirov} below in Section 3). 
Earlier Kalton \cite{Ka93} showed that in the case of symmetric sequence spaces with 
the Fatou property the $K$-monotonicity of a couple $(X, Y(w))$ for some
non-trivial weight $w$ implies that $X = l_p$ and $Y= l_q$ for some 
$1 \leq p, q \leq \infty$ (note, however, that there exist examples of shift-invariant sequence spaces $X$ with
the Fatou property, such that $(X,X(2^{-k}))$ is $K$-monotone but $X$ is not isomorphic to $l_p$
for any $1\le p\le \infty$ \cite{AT10a}, \cite{AT10b}).
Tikhomirov's theorem from \cite{Ti11} allows us to examine whether
the result of Kalton extends to symmetric function spaces. We will see
that this is not the case and the situation here will be essentially
different.

\smallskip
The paper is organized as follows. After the introduction, in Section 2, some necessary definitions 
and notations are collected. In the first part, we recall necessary information about symmetric 
spaces on $[0, 1]$ and then, in the second part, regularly varying convex Orlicz functions on 
$[0, \infty)$ and regularly varying quasi-concave functions on $[0, 1]$ are discussed.

In Section 3 we consider the notion of a $w$-decomposable Banach
lattice, which plays a central role in these investigations. Using
the Krivine theorem we show that it can be essentially simplified in
the case of symmetric function spaces. Namely, we prove condition
\eqref{decompForSym7} which means that for any $w$-decomposable symmetric 
space $X$ there exists $p \in [1,\infty]$ (depending on $X$) such that $X$ has, roughly
speaking, both "restricted lower and upper p-estimates". In
particular, its fundamental function $\varphi$ satisfies condition (\ref{condOnFundFunc}) 
for some $p$, which means that the function
$\varphi^p$ is "almost additive" near zero.

Section 4 contains results on the $w$-decomposability of Lorentz and
Marcinkiewicz spaces on $[0, 1]$. If $\varphi$ is a concave
increasing function on $[0, 1]$ with $\gamma_{\varphi} > 0$ and $1
\leq p < \infty$, then the couple $(X, X(w))$ with $X = \Lambda_{p,
\varphi}([0, 1])$ and a given non-trivial weight $w$  is
$K$-monotone if and only if condition (10) holds. This couple is
$K$-monotone for some weight $w$ if and only if $\varphi$ is
equivalent to a regularly varying function at 0 of order p.
Moreover, for any weight $w$ on $[0, 1]$ we can construct a concave
function $\varphi$ on $[0, 1]$ such that the couple $(X, X(w))$ with $X
= \Lambda_{1, \varphi}([0, 1])$ is $K$-monotone and $\Lambda_{1,
\varphi}([0, 1]) \neq L_1[0, 1]$. 

We obtain analogous results for Marcinkiewicz spaces, as a consequence 
of a new duality theorem which is of independent interest. It states that under 
suitable mild conditions on a Banach lattice $X$, the weighted couple
$(X, X(w))$ is $K$-monotone if and only if the couple $(X^{\prime}, X^{\prime}(w))$ 
is $K$-monotone, where $X^{\prime}$ means the K\"othe dual to $X$.

Section 5 deals with conditions of $w$-decomposability of Orlicz
spaces $L_F[0, 1]$. It is shown, in Theorem 6, that if an Orlicz
function $F$ satisfies the $\Delta_2$-condition for large arguments,
then $L_F[0, 1]$ is $w$-decomposable if and only if it satisfies some restricted 
$p$-upper and $p$-lower estimates (see condition (\ref{suffcondfororlicz})). 
Moreover, it is proved, in Theorem 7, that if an Orlicz
function $F$ is equivalent to an Orlicz function which is regularly
varying at $\infty$ of order $p\in [1,\infty)$, then the Orlicz
space $L_F = L_F[0, 1]$ is $w$-decomposable for some weight $w$ on
$[0, 1]$ and therefore the couple $(L_F, L_F(w))$ is $K$-monotone.

Finally, in Section 6, we prove that if a symmetric space $X$ on
$[0, 1]$ with non-trivial Boyd indices is $w$-decomposable with
respect to a weight changing sufficiently fast, then $X$ is an
ultrasymmetric Orlicz space. The result implies that, for such a weight $w$, 
every $K$-monotone couple $(X, X(w))$ with $X$ having the Fatou
property must be an ultrasymmetric Orlicz space. Moreover, if
its fundamental function is of the form $\varphi_X(t)=t^{1/p}$ for
some $1\leq p \leq \infty$, then $X=L_p$.
\vspace{3mm}

\begin{section}
{\bf Preliminaries}
\end{section}

Let us collect necessary information and results, in two parts, on
symmetric (rearrangement invariant) spaces and regularly varying
functions.

\medskip
{\bf 2a. Symmetric spaces.} Let $(\Omega, \Sigma, \mu)$ be a complete $\sigma$-finite measure space and 
$L^0=L^0(\Omega)$ be the space of all classes of $\mu$-measurable real-valued functions defined on 
$\Omega$.
A Banach space $X=\left( X,\| \cdot \|_{X}\right) $ is said to be a {\it Banach lattice} on $\Omega$ if $X$ is a linear 
subspace of $L^0(\Omega)$ and satisfies the so-called ideal property, which means that if $y\in X, x \in L^{0}$ 
and $|x(t)| \leq| y(t)|$ for $\mu $-almost all $t \in \Omega$, then $x\in X$ and $\| x\|_{X} \leq \| y\| _{X}$. 
We also assume that the support of the space $X$ is $\Omega$ (supp $X = \Omega$), that is, there is an element 
$x_0 \in X$ such that $x_0(t) > 0 ~ \mu$-a.e. on $\Omega$.

We will say that $X$ has the {\it Fatou property} if $0 \leq x_n \uparrow x \in L^0$ with $x_n \in X$ and 
$\sup_{n \in \mathbb N} \|x_n\|_X < \infty$ imply that $x \in X$ and $\|x_n\|_X \uparrow \|x\|_X$.

\medskip
A Banach lattice $X$ is said to be {\it $p$-convex} ($1 \leq p < \infty$), respectively {\it $q$-concave} ($1 \leq q < \infty$), 
if there is a constant $C > 0$ such that
$$
\| ( \sum_{k=1}^n |x_k|^p)^{1/p} \|_X \leq C ( \sum_{k=1}^n \| x_k\|_X^p )^{1/p},
$$
respectively,
$$
( \sum_{k=1}^n \| x_k\|_X^q )^{1/q} \leq C \| ( \sum_{k=1}^n |x_k|^q )^{1/q} \|_X,
$$
for any choice of vectors $x_1, x_2, \ldots, x_n$ in $X$ and any $n \in \mathbb N$. If in the above definitions 
vectors $x_1, x_2, \ldots, x_n \in X$ are pairwise disjoint, then $X$ is said to satisfy an upper $p$-estimate 
and lower $q$-estimate, respectively. Of course, $p$-convexity implies upper $p$-estimate and $q$-concavity 
implies lower $q$-estimate of a Banach lattice $X$. More properties can be found in the book \cite{LT79}.

\medskip
Let $w$ be a weight on $(\Omega,\Sigma,\mu)$, i.e., positive finite
a.e. function, and let $X$ be a Banach lattice on
$(\Omega,\Sigma,\mu)$. Then the weighted space $X(w)$ on
$(\Omega,\Sigma,\mu)$ is defined by $X(w) = \{x \in \Omega: x w \in
X\}$ with the norm $\|x\|_{X(w)} = \| x w\|_X$. In what follows, we
will always suppose that the weight $w$ is {\it non-trivial}, that is, $w$
or $1/w$ is an unbounded function on $(\Omega,\Sigma,\mu).$

For two Banach spaces $E$ and $F$ the symbol $E \stackrel {C} \hookrightarrow F$
means that the embedding $E \subset F$ is continuous with the norm
which is not bigger than $C$, i.e., $\|x\|_{F} \leq C \|x\|_{E}$ for
all $x \in E$.

\medskip
By a {\it symmetric space} (symmetric Banach function space), we
mean a Banach lattice $X = (X, \| \cdot\|_X)$ on $I = [0, 1]$ with
the Lebesgue measure $m$ satisfying the following additional property: for any
two equimeasurable functions $ x, y \in L^0(I)$ (that is, they have
the same distribution functions $d_x(\lambda) = d_y(\lambda)$, where
$d_x(\lambda) = m(\{t \in I: |x(t)| > \lambda \}), \lambda \geq 0$)
the condition $x \in X$ implies that $y \in X$ and $\| x \|_X = \| y
\|_X$. In particular, $\| x\|_X = \| x^\ast\|_X$, where $x^\ast
(t)={\rm inf}\{\lambda > 0\colon\ d_x(\lambda)\leq t\},\ t \geq 0$.

Recall that a non-negative function $\varphi: [0, 1] \rightarrow [0, \infty)$ is called {\it quasi-concave} if it is 
non-decreasing on $[0, 1]$ with $\varphi (0) = 0$ and if $\frac{\varphi(t)}{t}$ is non-increasing on $(0, 1]$.
The {\it fundamental function} $\varphi_X$ of a symmetric space $X$ on $I$ is defined by the formula 
$\varphi_X(t) = \| \chi_{[0, \,t]} \|_X, t \in I $. It is well known that every fundamental function is quasi-concave 
on $I$. Taking ${\tilde \varphi_X}(t): = \inf_{s \in (0, 1)} (1 + \frac{t}{s}) \varphi_X(s)$ we obtain 
a concave function ${\tilde \varphi_X}$ satisfying
$\varphi_X(t) \leq {\tilde \varphi_X}(t) \leq 2 \varphi_X(t)$ for all $t \in I$. 
For any quasi-concave function $\varphi$ on $I$ the {\it Marcinkiewicz space} $M_{\varphi}$ is defined by the norm
$$
\| x \|_{M_{\varphi}} = \sup_{t \in I, t > 0} \varphi(t) x^{**}(t), ~~ x^{**}(t) = \frac{1}{t} \int_0^t x^*(s) ds.
$$
This is a symmetric space on $I$ with the fundamental function $\varphi_{M_{\varphi}}(t) = \varphi(t)$ and 
$X \stackrel {1} \hookrightarrow M_{\varphi_X}$. The fundamental function of a symmetric space $X = (X, \| \cdot\|_X)$ 
is not necessarily concave but we can introduce an equivalent norm on $X$ in such a way that the fundamental 
function will be concave (take $ \| x \|_X^1 = \max (\| x \|_X, \| x \|_{M_{{\tilde \varphi_X}}}), ~ x \in X$).

For any symmetric function space $X$ with \konst{a} concave fundamental function $\varphi = \varphi_X$ there is also 
the smallest symmetric space with the same fundamental function. This space is the {\it Lorentz space} given by the norm
$$
\| x \|_{\Lambda_{\varphi}} = \int_I x^*(t) d \varphi(t): = \varphi(0^+) \| x \|_{L_{\infty}(I)} + \int_I x^*(t) \varphi^{\prime}(t) dt.
$$
We have then embeddings $\Lambda_{\varphi_X}  \stackrel {1} \hookrightarrow X  \stackrel {1} \hookrightarrow M_{\varphi_X}$. 
A non-trivial symmetric function space $X$ on $I = [0, 1]$ is an intermediate space between the spaces $L_1(I)$ and $L_{\infty}(I)$ 
and $L_{\infty}(I)  \stackrel {C_1} \hookrightarrow X  \stackrel {C_2} \hookrightarrow L_1(I)$,
where $C_1 = \varphi_X(1), C_2 = 1/\varphi_X(1)$ (see \cite{BS88}, Corollary 6.7 on page 78 or Theorem 4.1 on page 91 
of \cite{KPS82} for a similar result when the underlying measure space is $(0, \infty)$.)

\medskip
The lower and upper {\it Boyd indices} $\alpha_X$ resp. $\beta_X$ and the {\it dilation indices} $\gamma_X$ resp. $\delta_X$ 
of a symmetric space $X$ on $I = [0, 1]$ with the fundamental function $\varphi_X = \varphi$ are defined as follows:
$$
\alpha_X: = \lim_{t \rightarrow 0^+} \frac{\ln \| \sigma_t\|_{X\rightarrow X}}{\ln t}, ~ \beta_X: = 
\lim_{t \rightarrow \infty} \frac{\ln \| \sigma_t\|_{X\rightarrow X}}{\ln t}, ~ \sigma_t x(s) = x(s/t) \chi_I(s/t)
$$
and
$$
\gamma_X: = \gamma_{\varphi} = \lim_{t \rightarrow 0^+} \frac{\ln {\bar \varphi} (t)}{\ln t},  ~ \delta_X: = \delta_{\varphi} 
= \lim_{t \rightarrow \infty} \frac{\ln {\bar \varphi}(t)}{\ln t}, ~ {\bar \varphi} (t) = \sup_{s, st \in I} \frac{\varphi(st)}{\varphi(s)}.
$$
We have the relations $0 \leq \alpha_X \leq \gamma_X \leq \delta_X  \leq \beta_X \leq 1$ (see \cite{KPS82}, pp. 101-102 
and \cite{Ma85}, p. 28).

\medskip
A function $F: [0, \infty) \rightarrow [0, \infty)$ is called an {\it Orlicz function} if it is convex and increasing with $F(0) = 0$. 
For a given Orlicz function F the {\it Orlicz space} $L_F = L_F(I)$ on $I = [0, 1]$ is defined as
\vspace{-5mm}
\begin{equation*}
L_F(I) = \{x \in L^0(I): I_F(cx)<\infty ~{ \rm for ~some } ~ c = c(x) > 0\},
\end{equation*}
where $I_F(x): = \int_I F(|x(t)|) dt$. The Orlicz space $L_F$ is a symmetric space on $I$ with the so-called {\it Luxemburg-Nakano 
norm} defined by
\begin{equation*}
\| x\| _{L_F}=\inf \left\{ \lambda >0: I_{F} ( x/\lambda ) \leq 1\right\}.
\end{equation*}

An Orlicz function $F$ satisfies the {\it $\Delta_2$-condition for large $u$} if there exist constants $C \geq 1, u_0 \geq 0$ such 
that $F(2 u) \leq C F(u)$ for all $u \geq u_0$.

The following notation will be used throughout the text: $f
\stackrel{C}{\approx} g$ means that the functions $f$ and $g$ are
equivalent with the constant $C > 0$, that is, $C^{-1} f(t) \leq g(t) \leq Cf(t)$
for all points $t$ of the whole set on which these functions are defined, or 
at all points of some explicitly designated subset of that set. In the case when 
the constant of equivalence is not important we will write just $f
\approx g$. By $[r]$ we will denote the integer part of a real
number $r.$

More information about Banach lattices and symmetric spaces can be
found, for example, in \cite{BS88}, \cite{KPS82} and \cite{LT79};
about Orlicz spaces one can read e.g. in \cite{KR61} and \cite{Ma89}.

\bigskip
{\bf 2b. Regularly varying convex and concave functions.}  An Orlicz function $F$ on $[0,\infty)$ is called {\it regularly varying 
at $\infty$ of order $p$} ($1\leq p<\infty)$ if
\begin{equation} \label{def1}
\lim_{t\to\infty}\frac{F(tu)}{F(t)} = u^p \mbox{ for all } u > 0.
\end{equation}
The following result is due to Kalton \cite[Lemma 6.1]{Ka93}.
\vspace{3mm}

\begin{lemma} \label{lemma1}
Let $F$ be an Orlicz function. The following three conditions are equivalent:
\begin{itemize}
\item[$(a)$] F is equivalent to a regularly varying  Orlicz function at $\infty$ of order $p \in [1, \infty)$.
\item[$(b)$] There exists a constant $C > 0$ such that for any $u \in (0, 1]$ we can find $t_0 = t_0(u)$ with
\begin{equation*}
\frac{F(tu)}{F(t)} \stackrel{C}{\approx} u^p ~~{\it for ~all} ~~t \geq t_0.
\end{equation*}
\end{itemize}
\end{lemma}

Although we do not need it here, there is an analogous definition to the one above for Orlicz functions which are 
regularly varying of order $p$ at $0$ instead of at $\infty$ (see e.g. \cite{Ka93}). However, we do need to consider 
quasi-concave functions which are regularly varying of order $p$ at $0$. Before recalling the definition of these 
we should point out that it is not quite analogous to the definitions for regularly varying Orlicz functions, because
the power $p$ which appeared in (\ref{def1}) and in the corresponding definition in \cite{Ka93} will be replaced in
(\ref{def2}) by the power $1/p$.

A function $\varphi: [0, 1] \rightarrow [0, \infty)$ which is quasi-concave and satisfies $\varphi(0)=0$ is said to be
{\it regularly varying at zero of order $p$} $(1\leq p\leq \infty)$ if
\begin{equation} \label{def2}
\lim_{t\to 0^+}\frac{\varphi(t u)}{\varphi(t)} = u^{1/p}\;\;\mbox{for all}\;\;u>0.
\end{equation}

Abakumov and Mekler \cite[Theorem 5]{AM94} proved that a quasi-concave function $\varphi$ is equivalent to 
a quasi-concave regularly varying function at zero of order $p\in [1,\infty]$ if and only if
$$
\limsup_{t\to 0^+} \frac{\varphi(tu)}{\varphi(t)}\approx u^{1/p} ~~ {\rm for ~all} ~ u>0.
$$
The following lemma is an immediate consequence of this result (see also the proof of
Theorem 5 in \cite{AM94}).

\begin{lemma} \label{lemma2}
A quasi-concave function $\varphi$ on $[0, 1]$ is equivalent to a quasi-concave 
function which is regularly varying at zero of order $p\in [1,\infty]$ if and only
if for some $C > 0$ and any $N\in \mathbb N$ there exists $\tau(N)\in(0,1]$ such
that  for all $0< t \leq \tau(N), 0 < t N \leq 1$ we have
\begin{equation}\label{regVarEquivalent}
\frac{\varphi(Nt)}{\varphi(t)}  \stackrel{C}{\approx} N^{1/p}.
\end{equation}
\end{lemma}

Recall that the fundamental function of an Orlicz space $L_F$ on $[0,1]$ with the Luxemburg-Nakano
norm is $\varphi_{L_F}(t) = \frac{1}{F^{-1}(1/t)}$ for $0<t\leq 1$ and $\varphi_{L_F}(0) = 0$, 
where $F^{-1}$ is the inverse of $F$ (see formula (9.23) in \cite{KR61} on page 79 of the English version or 
Corollary 5 in \cite{Ma89} on page 58). The function $\varphi_{L_F}$ is quasi-concave but not necessarily 
concave on $[0, 1]$ (see \cite{KR61} or \cite{Ma89}).

The notions of regularly varying Orlicz and quasi-concave functions are closely interrelated. 
Using Lemmas  \ref{lemma1} and \ref{lemma2} and routine arguments we establish the following 
quantitative result showing a connection between an regularly varying Orlicz function $F$ and 
the fundamental function of the corresponding Orlicz space $L_F$.

\begin{proposition} \label{equiv of functions}
Suppose that $p \in [1, \infty)$ and let $F$ be an Orlicz function such that both $F$ and its complementary 
function $F^{*}$ satisfy the $\Delta_2$-condition for large $u$. Then the following conditions are equivalent:
\begin{itemize}
\item[$(a)$] There exists a constant $C^{\prime} > 0$ such that for any $N\in \mathbb N$ there exists $\tau(N)\in(0,1]$ with
\begin{equation}\label{regVarEquivalentConvex}
\frac{F(u)}{F(uN^{-1/p})}  \stackrel{C^{\prime}}{\approx} N ~ {\rm for ~all} ~ u \geq F^{-1}(1/\tau(N)).
\end{equation}
\item[$(b)$] There exists a constant $C > 0$ such that for any $N\in\mathbb N$ the fundamental function $\varphi_{L_F}$ satisfies condition
\eqref{regVarEquivalent} with the same $\tau(N).$
\end{itemize}
\end{proposition}

\begin{section}
{\bf $w$--decomposable Banach lattices}
\end{section}

Later on $C$ will denote a constant whose value may be different in its different appearances.

The following notion was introduced in paper \cite{Ti11} and it will be very important for us.
Let $X$ be a Banach lattice on $(\Omega,\Sigma,\mu)$ and $w$ be a weight on $\Omega$.
We say that $X$ is {\it $w$-decomposable} if there exists $C > 0$ such that for any $n\in \mathbb N$ and for all 
$x_1,\dots,x_n,$ $y_1,\dots,y_n$ in $X$ satisfying the conditions:
\begin{equation}\label{equalnorms}
\|x_i\|_X = \|y_i\|_X,  ~ i =1,2, \dots, n,
\end{equation}
and
\begin{equation}\label{arrangement}
\inf w({\rm supp}\, x_i \cup {\rm supp}\, y_i) \geq 2 \sup w({\rm supp}\, x_{i+1} \cup {\rm supp}\, y_{i+1}),  ~ i =1,2, \dots, n-1,
\end{equation}
we have that
\begin{equation}\label{equalsums}
\|\sum\limits_{i=1}^{n}x_i \|_X  \stackrel{C}{\approx} \|\sum\limits_{i=1}^{n}y_i \|_X.
\end{equation}

To clarify the meaning of condition \eqref{arrangement}, consider the following example:
let $X$ be a Banach lattice of Lebesgue measurable functions on $[0,1]$ and
$w(t)=1/t$ $(0<t\leq 1)$. Then \eqref{arrangement} is equivalent to the following inequality
$$2\sup({\rm supp}\, x_i\cup {\rm supp}\,y_i )\leq \inf({\rm supp}\, x_{i+1}\cup {\rm supp}\,y_{i+1} ),  ~ i =1,2, \dots, n-1.$$
In other words, there are some intervals $[a_i, b_i]\subset[0, 1]$ (depending on $x_i$, $y_i$) such that
$2b_i\leq a_{i+1}$ $(i=1,2,\dots,n-1), {\rm supp}\, x_i\subset[a_i, b_i]$ and ${\rm supp}\,y_i\subset[a_i, b_i]$ $(i=1,2,\dots,n)$.

It is not hard to see that $1/t$-decomposability is equivalent to $1/t^q$-decomposability and, more generally,
$w$-decomposability and $w^q$-decomposability are equivalent for any weight $w$ and any $q>0$
(see \cite{Ti11b}, Corollary~2.2 on page~61).

It turns out that the $w$-decomposability of a Banach lattice $X$ guarantees the $K$--monotonicity
of the weighted couple $(X, X(w))$. More precisely, Tikhomirov in \cite{Ti11} obtained the following
generalization of Kalton's results from \cite{Ka93}.

\vspace{3mm}
\begin{theorem}\label{Theor Tikhomirov}
{\it  Suppose $X$ is a
Banach lattice on a $\sigma$--finite measure space $(\Omega, \Sigma,
\mu)$ with supp $X = \Omega$ which has the Fatou property and $w$ is a (non-trivial) weight on
$\Omega$. Then the Banach couple $(X, X(w))$ is $K$--monotone if and
only if $X$ is $w$-decomposable.
}
\end{theorem}

In the case of symmetric spaces on $[0,1]$ the notion of
$w$-decomposability can be clarified by using the well--known Krivine
theorem.

\begin{proposition} \label{prop2}
Let $w$ be a weight on $[0,1]$. A symmetric space
$X$ on $[0,1]$ is $w$-decomposable if and only if there exist $C >0$ and $1\leq
p\leq\infty$ such that for any $n\in\mathbb N$ and for all
$x_1, x_2, \dots, x_n\in X$ satisfying the conditions
\begin{equation}\label{decompForSym}
\inf w({\rm supp}\, x_i)\ge 2 \sup w({\rm supp}\, x_{i+1}),\; 1 \leq i \leq n-1,
\end{equation}
we have that
\begin{equation} \label{decompForSym7}
\left\|\sum_{i=1}^{n}x_i\right\|_X ~ \stackrel{C}{\approx} ~ \left(\sum_{i=1}^{n} \|x_i\|_X^p\right)^{1/p},
\end{equation}
where, as usual, in the case $p=\infty$ the right hand side
should be replaced by $\max_{1 \leq i \leq n}  \|x_i\|_X$.
\end{proposition}

\begin{proof}
By Krivine's theorem (see \cite[Theorem 2.b.6]{LT79} or \cite{Ro78}), there exists $p\in [1/\beta_X, 1/\alpha_X]$ 
such that for every $m\in\mathbb N$ there
are pairwise disjoint equimeasurable functions $y_1, y_2, \dots, y_m$ $\in X, \|y_k\|_X = 1$ $(k=1, 2, \dots, m),$ 
such that for any $\alpha_k \in \mathbb R$ $(k=1, 2, \dots, m)$ we have
\begin{equation}\label{Kriv Theorem}
\frac12\|(\alpha_k)\|_p\le\Big\|\sum_{k=1}^m \alpha_ky_k\Big\|_X\le
2\|(\alpha_k)\|_p.
\end{equation}
Obviously, the support of each function $y_k$ has measure not greater than $1/m$.

Suppose that a symmetric space $X$ is $w$-decomposable and that, for some $n\in\mathbb N$,
functions $x_1,\dots,x_n$ in $X$ satisfy condition \eqref{decompForSym}.
Without loss of generality we may assume that $x_i\ne 0$ for each $i = 1, 2, \dots, n$. We
choose $m \in \mathbb N$ sufficiently large so that the support of each $x_i$ has measure greater 
than $1/m$ (and so of course we also have $m \geq n$). For this choice of $m$ we consider the disjoint 
measurable functions $y_1, y_2, \ldots, y_m, \|y_k\|_X = 1 \, (k=1, 2, \dots, m)$, obtained as it is described 
in the previous paragraph. In fact, we will only need the first $n$ of these functions, and we will only 
need special case of (\ref{Kriv Theorem}) for sequences $(\alpha_k)$ which satisfy $\alpha_k = 0$ 
for $k > n$. We may assume without loss of generality, that the support of $y_i$ is contained in the 
support of $x_i$ for each $i = 1, 2, \ldots, n$. (If not, since $X$ is symmetric, we can simply replace each 
$y_i$ by an equimeasurable function which has this property and the above mentioned special 
case of (\ref{Kriv Theorem}) will remain valid.) Thus condition (\ref{decompForSym}) implies that 
condition (\ref{arrangement}) is satisfied and therefore, applying $w$-decomposability (see (\ref{equalsums})) 
and then the special case of (\ref{Kriv Theorem}), we obtain that
$$
\Big\|\sum_{i=1}^n \alpha_i\frac{x_i}{\|x_i\|_X}\Big\|_X \approx
\Big\|\sum_{i=1}^n \alpha_i y_i \Big\|_X \approx \|(\alpha_k)_{i=1}^n\|_p
$$
for all choices of real numbers $\alpha_i$. In particular, when $\alpha_i = \|x_i\|_X$ we obtain 
\eqref{decompForSym7}. Since the reverse implication is obvious, the proof is complete.
\end{proof}

For a given weight $w$ consider the sets
$$
M_k:=\{t\in[0,1]:w(t)\in[2^k, 2^{k+1})\}, k\in\mathbb Z.
$$
Let $(w_r)_{r=1}^{\infty}$ be the non-increasing rearrangement of
the sequence $(m (M_k))_{k=-\infty}^{+\infty}.$ Since the
weight $w$ is non-trivial it follows that $w_r>0$ for all $r=1, 2, \dots$.

For some fixed $n\in\mathbb N$, let $x_1, x_2, \dots, x_n$ be functions in $X$. Suppose first that 
these $n$ functions satisfy condition \eqref{decompForSym}. Then it is easy to see that 
$$
{\rm card}\{i:\,M_k\cap {\rm supp}\, x_i\ne\emptyset\} \leq 1 ~~ {\rm for ~ each} ~~ k\in \mathbb Z.
$$ 
Alternatively, more or less conversely, suppose that the functions $x_i$ satisfy
$$
{\rm card}\{k:\,M_k\cap {\rm supp}\, x_i\ne\emptyset\} \leq 1 ~~ {\rm for ~ each} ~~i\in \{1, 2, \ldots, n\},
$$
i.e., for each $i$, there exists a unique $k_i \in \mathbb Z$ for which ${\rm supp}\, x_i \subset M_{k_i}$. 
Furthermore, suppose $k_1 < k_2 < \ldots < k_n$. While this is not sufficient to imply that the collection 
of functions $x_1, x_2, \ldots, x_n$ satisfies condition \eqref{decompForSym}, it does imply that (after 
relabelling) the collection of functions $x_1, x_3, x_5, \dots$ satisfies  \eqref{decompForSym} and so 
does the collection $x_2, x_4, \dots$. 

By $\{\overline{M}_r\}_{r=1}^{\infty}$ we will denote any rearrangement of the sets $M_k \, (k = 0, \pm 1,
\pm 2, \ldots)$ such that $m(\overline{M}_r) = w_r, r = 1, 2, \ldots$. Thus, by Proposition \ref{prop2}, we obtain 
the following result.

\begin{theorem} \label{thm 1}
Suppose $w$ is a non-trivial weight on $[0,1]$. A symmetric space
$X$ on $[0,1]$ is $w$-decomposable if and only if there exist $C >0$
and $1\le p\le\infty$ such that for any $n\in\mathbb N$ and
for all $x_1, x_2, \dots, x_n\in X$ satisfying the condition
\begin{equation}\label{inclForSupp}
{\rm supp}\, x_i\subset \overline{M}_i,\;\;1\le i\le n,
\end{equation}
we have \eqref{decompForSym7}. 
\end{theorem}

Next, we will need some corollaries of Theorem \ref{thm 1}. Firstly, using the symmetry of the norm in $X$, we get

\begin{corollary}\label{cor simplifies thm 3.3}
Let $w$ be a non-trivial weight on $[0,1]$. A symmetric space
$X$ on $[0,1]$ is $w$-decomposable if and only if there exist $C >0$
and $1\le p\le\infty$ such that for any $n\in\mathbb N$ and
for all pairwise disjoint $x_1, x_2, \dots, x_n\in X$ satisfying the condition
\begin{equation}\label{inclForSupp}
m({\rm supp}\, x_i)\leq w_i,\;\;1\le i\le n,
\end{equation}
we have \eqref{decompForSym7}.
\end{corollary}

\begin{corollary}\label{cor1}
A symmetric space $X$ on $[0,1]$ is $w$-decomposable for some non-trivial weight $w$ 
on $[0,1]$ if and only if there exist $C>0$, $1\leq p\leq \infty$, and a sequence
of disjoint intervals $\{\Delta_k\}_{k=1}^\infty$ from $[0,1]$ such
that for any $n\in\mathbb N$ and for all $x_1, x_2, \dots, x_n\in X$ satisfying
the condition ${\rm supp}\, x_i\subset \Delta_i$ $(1\leq i\leq n)$ we have 
\eqref{decompForSym7}.
\end{corollary}

\begin{corollary}\label{cor1new}
Let $w$ be a non-trivial weight on $[0,1]$ and let the sequence  $(w_r)_{r=1}^{\infty}$ 
be as above. Suppose that $X$ is a $w$-decomposable symmetric space $X$ on $[0,1]$ 
with the fundamental function $\varphi$. Then there exist some $C > 0$ and $p \in [1, \infty]$ such 
that, for every sequence of reals $(\tau_r)_{r=1}^{\infty}$ satisfying $0<\tau_r\le w_r$ $(r \in \mathbb N)$,
we have 
\begin{equation}\label{condOnFundFunc}
\varphi\bigl(\sum_{r=1}^{\infty}\tau_r\bigr) \stackrel{C}{\approx}
\left(\sum_{r=1}^{\infty}\varphi^p(\tau_r)\right)^{1/p}
\end{equation}
with the natural modification for $p=\infty$.
\end{corollary}

\begin{corollary}\label{cor2}
Let $w$ be a non-trivial weight on $[0,1]$ such that a symmetric
space $X$ on $[0,1]$ is $w$-decomposable. Then there exist $C>0$ and
$1\leq p\leq\infty$ such that condition \eqref{regVarEquivalent} is
fulfilled with $\tau(N)=w_N$ $(N\in\mathbb N)$. In particular, the
fundamental function $\varphi$ of $X$ is equivalent to a regularly
varying function at zero of order $p$ and
$\alpha_X=\gamma_\varphi=\delta_\varphi=\beta_X=1/p.$
\end{corollary}

\begin{proof}
First we note that condition \eqref{regVarEquivalent} is an immediate
consequence of \eqref{condOnFundFunc}. Moreover, it is well known
that the assertion of Krivine's theorem holds for both $p=1/\alpha_X$ 
and $p=1/\beta_X$ (see \cite[p. 141]{LT79}, \cite{Ro78} and \cite{As11}). 
Therefore, coincidence of the Boyd indices and dilation indices follows 
from an inspection of the proof of Proposition \ref{prop2} and the
inequalities $\alpha_X \leq \gamma_\varphi\leq \delta_\varphi \leq
\beta_X$ (cf. \cite[p. 102]{KPS82} and \cite[p. 28]{Ma85}).
\end{proof}

Let us show that, conversely, \eqref{condOnFundFunc} can be derived
from \eqref{regVarEquivalent} with $\tau(N)=w_N$ for a large class
of weights $w$.

\begin{theorem}\label{ExponentM}
Let $w$ be a weight on $[0, 1]$ such
that $q w_{r+1}\le w_r$ $(r=1, 2, \dots)$ for some $q>1$ and let $\varphi$ be a
quasi--concave function on $[0,1].$
Suppose there exist $C>0$ and $1\leq p \leq\infty$ such that
$\varphi$ satisfies \eqref{regVarEquivalent} with $\tau(N)=w_N$
($N=1,2,\dots).$ Then, for any sequence of reals
$(\tau_r)_{r=1}^{\infty}$ such that $0<\tau_r\le w_r$
$(r=1,2,\dots)$, estimate \eqref{condOnFundFunc} holds.
\end{theorem}

\begin{proof}
We present the proof for $1\leq p<\infty$ since the case $p=\infty$ needs only minor changes.

Firstly, it is easy to see that condition \eqref{regVarEquivalent} can be extended
as follows: we can find a (possibly different) constant $ C > 0$
such that for every real $z \geq 1$ and $\tau(z):=\tau ([z])$ we have
\begin{equation} \label{equiv14}
\frac{\varphi(z t)}{\varphi(t)} ~ \stackrel{C}{\approx} ~ z^{1/p}\;\;\mbox{if} \;\; 0 < t \leq \tau(z).
\end{equation}

Let us show that for every $m\in\mathbb N$ there is a
constant $C(m) > 0$ such that for all even $N \in \mathbb N$ satisfying the 
inequality $N^m \leq q^{N/2}$ and all $z \in [1, N]$ we have
\begin{equation}\label{regVarEqImprove}
\frac{\varphi(z^m t)}{\varphi(t)} ~ \stackrel{C(m)}{\approx} ~ z^{m/p}\;\;\mbox{if } \;\;0 < t \leq\tau(N).
\end{equation}
In fact, by the assumption, $\tau(N/2) \geq q^{N/2} \tau(N)$, whence
$$z^{k} t \leq z^m t \leq N^m \tau(N) \leq q^{N/2} \tau(N) \leq \tau(N/2) \leq 1 ~(k = 0, 1, \ldots, m)$$
provided that $t \leq \tau(N)$. Therefore, using the quasi--concavity of $\varphi$
and equivalence \eqref{equiv14} for $\max(1,z/2)$ we obtain that
\begin{equation*}\label{regVarEqImprove2}
\frac{\varphi(z^k t)}{\varphi(z^{k-1}t)} \approx \frac{ \varphi(\max(1,z/2) z^{k-1} t)}{\varphi(z^{k-1} t)} 
\approx z^{1/p}\;\;\mbox{if }\;\;0<t \leq \tau(N),
\end{equation*}
with a constant of equivalence depending on $p.$ Multiplying these relations
for all $k=1, 2, \dots, m$, we come to \eqref{regVarEqImprove}.

Next, let
$$
\overline{\varphi}^0 (s) = \limsup_{t \rightarrow 0^+} \frac{\varphi(ts)}{\varphi(t)} ~ {\rm for } ~~ s > 0.
$$
Clearly, condition (\ref{equiv14}) implies $\overline{\varphi}^0 (s) \approx s^{1/p} (s > 0)$. 
On the other hand, in view of Boyd's result \cite{Bo71} (see also \cite[Theorem 2.2]{Ma85})
$\overline{\varphi}^0 (s) \geq s^{\gamma_{\varphi}}$ if $0 < s \leq 1$ and 
$\overline{\varphi}^0 (s) \geq s^{\delta_{\varphi}}$ if $ s > 1$.
Since $\gamma_{\varphi} \leq \delta_{\varphi}$ it follows that $\gamma_{\varphi} = \delta_{\varphi} = \frac{1}{p} > 0$. 
Therefore, there exist $A > 0$ and $\kappa>0$ such that
\begin{equation}\label{positive index}
\sup\limits_{0<s \le 1}\frac{\varphi(st)}{\varphi(s)}\le At^{\kappa} ~~{\rm for ~ all} ~~ 0\le t\le 1.
\end{equation}

Let us prove that \eqref{condOnFundFunc} is a consequence of \eqref{regVarEqImprove} and 
\eqref{positive index}. Take a natural number $m_0 \geq 2$ such that $\kappa \,m_0 > 1$ and 
consider an arbitrary sequence $(\tau_r)_{r=1}^{\infty}$ satisfying $\tau_r \leq w_r, r = 1, 2, \dots$.
Since the non-increasing rearrangement $(\tau^*_r)_{r=1}^{\infty}$ of this sequence also satisfies 
$\tau^*_r \leq w_r$ for $ r = 1, 2, \dots$ we can assume without loss of generality that the sequence 
$(\tau_r)_{r=1}^{\infty}$  is itself non-increasing. Further, set 
$I=\{r \in {\mathbb N}: \tau_r \, r^{m_0} \geq \tau_1 \}, J = {\mathbb N}\setminus I.$ 
Clearly, $1\in I$. By \eqref{positive index} and the choice of $m_0$,
$$
\varphi \Big(\sum_{r\in J}\tau_r \Big) \leq \varphi\Big(\sum_{r=2}^{\infty}\frac{\tau_1}{r^{m_0}}\Big) \leq
A\Big(\sum_{r=2}^{\infty}r^{-m_0} \Big)^{\kappa} \,\varphi(\tau_1)\leq C_1\varphi(\tau_1).
$$
Analogously,
$$
\sum_{r\in J}\varphi^p(\tau_r)\le\sum_{r=2}^{\infty} \varphi^p(\tau_1/r^{m_0})\le A^p\sum_{r=2}^\infty {r^{-p \,\kappa
\, m_0}}\varphi^p(\tau_1) \le C_2 \, \varphi^p(\tau_1).
$$
Thus, it is sufficient to prove equivalence \eqref{condOnFundFunc} for $(\tau_r)_{r\in I}.$

If ${\rm card} \,I < \infty$ then there is nothing to prove. So, assume that ${\rm card} \,I = \infty$.
Choose a positive integer $i_0 \in I$, $i_0\geq 2$ such that for $N=2\,[i_0/2]$ we have $N^{m_0} \leq q^{N/2}$.
Denote $\delta_r = (\tau_r/\tau_{i_0})^{1/m_0}$ for $r\in I \cap \{1, 2, \ldots, i_0\}$. Then,
by the definition of $I$, $\delta_r \leq (\tau_1/\tau_{i_0})^{1/m_0} \leq i_0\leq 2N$.
Applying \eqref{regVarEqImprove} in the case $m=m_0, z = \max(1,\delta_r/2)$ for all $r \in I, r \leq i_0$, we get
\vspace{-2mm}
$$
\frac{\varphi(\tau_r)}{\varphi(\tau_{i_0})} =\frac{\varphi(\delta_r^{m_0}\tau_{i_0})}{\varphi(\tau_{i_0})} ~
{\approx} ~ \delta_r^{m_0/p}=\left(\frac{\tau_r}{\tau_{i_0}}\right)^{1/p},
$$
with a constant of equivalence depending on $m_0$ and $p.$ The last formula implies that
\vspace{-2mm}
$$
\sum\limits_{r\in  I \cap \{1, 2, \ldots, i_0\}}\varphi^p(\tau_r) ~ {\approx} 
~ \frac{\varphi^p(\tau_{i_0})}{\tau_{i_0}}\sum_{r\in  I \cap \{1, 2, \ldots, i_0\}}\tau_r.
$$
On the other hand, setting $\delta: = \left(\sum_{r \in I \cap \{1, 2, \ldots, i_0\}}\tau_r/\tau_{i_0}\right)^{1/(m_0+1)}$ 
we get
\vspace{-2mm}
$$
\delta \leq \Big(\sum_{r \in I \cap \{1, 2, \ldots, i_0\}}\tau_1/\tau_{i_0} \Big)^{1/(m_0+1)} \leq i_0.
$$
Therefore, again by (\ref{regVarEqImprove}), we obtain
$$
\frac{\varphi^p(\tau_{i_0})}{\tau_{i_0}}\sum_{r\in I \cap \{1, 2, \ldots, i_0\} }\tau_r 
= \delta^{m_0+1} \varphi^p(\tau_{i_0}) ~ {\approx} ~
\varphi^p(\delta^{m_0+1}\tau_{i_0}) = \varphi^p\Big(\sum\limits_{r\in I \cap \{1, 2, \ldots, i_0\} }\tau_r \Big),
$$
with a constant depending on $m_0$ and $p.$ Combining the above
formulas and noting that $i_0$ can be arbitrarily large,
we conclude that equivalence \eqref{condOnFundFunc} holds and the proof is complete.
\end{proof}

Theorem~\ref{ExponentM} allows us to construct non-trivial quasi--concave functions 
satisfying condition \eqref{condOnFundFunc}, for a large class of weights.
For example, let $w(t)=1/t$ $(0< t\leq 1)$. In this case $w_r=2^{-r}$, $r=1,2,\dots$
Define $\varphi(t)=t\,\log\frac{e}{t}$ $(0< t\leq 1)$. Obviously, $\varphi$ is quasi--concave.
Elementary calculations show that \eqref{regVarEquivalent} is fulfilled for $\varphi$ with $p=1$ and $\tau(N)=w_N=2^{-N}$
($N=1,2,\dots).$ Thus, by Theorem~\ref{ExponentM}, $\varphi$ satisfies \eqref{condOnFundFunc}.

\begin{section}
{\bf $w$--decomposable Lorentz and Marcinkiewicz spaces}
\end{section}

For $1 \leq p < \infty$ and any increasing concave function
$\varphi,$ $\varphi(0)=0,$ the Lorentz space $\Lambda_{p,\varphi}$
consists of all classes of measurable functions $x$ on $[0,1]$ such that
$$
\|x\|_{\Lambda_{p,\varphi}} = \left( \int_{0}^{1} \left[ x^{*}(t)
\varphi(t) \right ]^{p} \frac{dt}{t} \right)^{1/p}<\infty.
$$
The space $\Lambda_{p,\varphi}$ was investigated by Sharpley
\cite{Sh72} and Raynaud \cite{Ra92}, who proved that if $0 <
\gamma_{\varphi} \leq \delta_{\varphi} < 1$, then $\Lambda_{p,\varphi}$ 
is a symmetric space on $[0,1]$ with an equivalent norm
$$
\|x\|_{\Lambda_{p,\varphi}}^{\star} = \left( \int_{0}^{1} \left[
x^{**}(t) \varphi(t) \right ]^{p} \frac{dt}{t} \right)^{1/p},
$$
where $x^{**}(t)= \frac{1}{t} \int_0^t x^*(s)\,ds$ (cf. \cite{Sh72},
Lemma 3.1). Moreover, if $\gamma_{\varphi} > 0,$ then applying
Corollary 3 on page 57 of \cite{KPS82} to the function $\psi = \varphi^p$ $(1\le p<\infty)$ 
(see also \cite[Theorem 6.4(a)]{Ma85}), we obtain that there exists a constant
$K = K(p) \geq 1$ such that
\begin{equation}\label{Cor of KPS}
K^{-1} \varphi^p(t) \leq \int_0^t\frac{\varphi^p(s)}{s}\,ds \leq K
\varphi^p(t)\;\;(0<t\le 1).
\end{equation}
Therefore, the fundamental function $\varphi_{\Lambda_{p,\varphi}}(t)$ is equivalent
to $\varphi(t)$. Inequalities \eqref{Cor of KPS} imply also that, if
$\gamma_{\varphi}
> 0$, then the space $\Lambda_{1,\varphi}$ coincides with the Lorentz space
$\Lambda_{\varphi}$ with the norm
\vspace{-2mm}
$$
\|x\|_{\Lambda_{\varphi}}:=\int_0^1 x^*(t)d\varphi(t).
$$
Recall also that the K\"othe dual of the Lorentz space $\Lambda_{\varphi}$ is isometric
to the Marcinkiewicz space $M_{\tilde{\varphi}}$ with $\tilde{\varphi}(t) = \frac{t}{\varphi(t)}$ and
its norm is
$$
\|x\|_{M_{\tilde{\varphi}}}=\sup_{0<t\le 1} \tilde{\varphi} (t)
x^{**}(t) =  \sup_{0<t\le 1} \frac{1}{\varphi(t)} \int_0^t x^*(s)ds
$$
(cf. \cite{KPS82},  Theorem 5.2 on page 112).

\medskip
We will prove that condition \eqref{condOnFundFunc} is necessary
and sufficient for Lorentz and Marcin\-kie\-wicz spaces to be
$w$-decomposable. We start by proving a specific geometric property
of Lorentz spaces.

\vspace{3mm}
\begin{proposition}\label{LorentzMainProp}
Let $\varphi$ be an increasing non-negative concave function on $[0,1]$ such that $\gamma_{\varphi}>0$, and let 
$1\le p<\infty.$ Then for arbitrary $b>1$ there exists a constant $C = C(b,\varphi, p) > 0$ with the following property:
for any two-sided non-decreasing sequence $(a_j)_{j=-\infty}^{+\infty}$ of reals from $[0,1]$ such that the function
$x=\sum_{j=-\infty}^{+\infty}b^{-j}\chi_{(a_{j-1}, a_j]}$ belongs to $\Lambda_{p,\varphi}$, 
we have
\begin{equation}\label{LorentzMainIneq}
\|x\|_{\Lambda_{p,\varphi}}^p ~ \stackrel{C}{\approx} ~
\sum_{j=-\infty}^{+\infty}b^{-pj}\varphi^p (a_{j}-a_{j-1}).
\end{equation}
\end{proposition}

\begin{proof}
Since $\gamma_{\varphi}>0$, there exist $\kappa > 0$ and $ A > 0$ such that inequality 
(\ref{positive index}) holds. Choose a constant $C_1 = C_1(\varphi) > 1$ 
satisfying the inequality
\begin{equation} \label{estimate19}
\frac{(C_1 +1)^{\kappa}}{A} \geq 2 \, K^2,
\end{equation}
where $K$ is the constant from \eqref{Cor of KPS}, and denote by $I$ the set of all indices 
$j \in \mathbb Z$ such that $a_j - a_{j-1} \geq C_1\, a_{j-1}$.
We prove the following equivalences:
\begin{equation} \label{estimate20}
\int_{a_{j-1}}^{a_j}\frac{\varphi^p(t)}{t}\,dt \approx \varphi^p(a_j-a_{j-1}), ~ j \in I
\end{equation}
and, if $b > C_1 + 1$,
\begin{equation} \label{estimate21}
\|x\|_{\Lambda_{p,\varphi}}^p \approx \sum_{j \in I} b^{-p j} \int_{a_{j-1}}^{a_j}\frac{\varphi^p(t)}{t}\,dt ,
\end{equation} 
\begin{equation} \label{estimate22}
\sum_{j=-\infty}^{+\infty}b^{-pj}\varphi^p (a_{j}-a_{j-1}) ~ {\approx} ~ \sum_{j\in I}b^{-pj}\varphi^p (a_{j}-a_{j-1}),
\end{equation}
with constants which depend only on $b, \varphi$ and $p$.

At first, if $j \in I$ then, by (\ref{positive index}) and (\ref{estimate19}), 
$$
\varphi(a_j) \geq \varphi((C_1+1) a_{j-1}) \geq \frac{(C_1+1)^{\kappa}}{A} \varphi(a_{j-1}) 
\geq 2 K^2\, \varphi(a_{j-1}).
$$
Combining this with (\ref{Cor of KPS}) and the inequality
\begin{equation} \label{estimate23}
\varphi(a_j) \leq \varphi (a_j - a_{j-1}) + \varphi(a_{j-1}) \leq 2 \, \varphi(a_j - a_{j-1}),
\end{equation}
we obtain 
\vspace{-2mm}
\begin{eqnarray*}
\frac{1}{2K} \, \varphi^p(a_j &- & a_{j-1}) \leq
\frac{1}{2K} \, \varphi^p(a_j) \leq \frac{1}{2K} [2\, \varphi^p(a_j) - (2 K^2)^p\, \varphi^p(a_{j-1})]\\
&\leq&
 \frac{1}{2K} [2\, \varphi^p(a_j) - 2 K^2\, \varphi^p(a_{j-1})] =  \frac{\varphi^p(a_j)}{K} - K\, \varphi^p(a_{j-1})\\
 &\leq&
 \int_{0}^{a_j}\frac{\varphi^p(t)}{t}\,dt - \int_{0}^{a_{j-1}} \frac{\varphi^p(t)}{t}\,dt = \int_{a_{j-1}}^{a_j}\frac{\varphi^p(t)}{t} \, dt\\
&\leq& 
 \int_{0}^{a_j}\frac{\varphi^p(t)}{t}\,dt \leq K\varphi^p(a_j) \leq 2^{p}K\varphi^p(a_j-a_{j-1}),
\end{eqnarray*}
which implies (\ref{estimate20}).

Now, assuming $b > C_1 + 1$, we show that the set $I$ is unbounded from below. In fact, otherwise there 
is $j_0 \in \mathbb Z$ such that $a_j - a_{j-1} < C_1 a_{j-1}$ for all $j \leq j_0$. Then, we have 
$a_{j_0} \leq (C_1+1)^{j_0-j}\, a_j$ $(j \leq j_0)$ and by (\ref{Cor of KPS}) and the concavity of $\varphi$,
\begin{eqnarray*}
\|x\|_{\Lambda_{p,\varphi}}^p 
&\geq& 
\sup_{j\leq j_0} b^{-pj}\int_0^{a_j}\frac{\varphi^p(t)}{t}\,dt 
\geq \frac{1}{K} \, \sup_{j\leq j_0} b^{-pj} \varphi^p(a_j) \\
&\geq&
\frac{1}{K} \, \sup_{j\leq j_0} \frac{(C_1+1)^{p(j-j_0)}}{b^{pj}} \, \varphi^p(a_{j_0}) = \infty.
\end{eqnarray*}
Therefore, for a given $i \notin I$ we can find $k = \max \{j < i: j \in I \}$. Further, from the definition 
of $I$ it follows that $ a_i <  (C_1 + 1)^{i-k} \, a_k.$
Since $\varphi$ is concave and $2 a_{k-1} \leq a_k$, we get
\vspace{-2mm}
\begin{eqnarray*}
\int_{a_{i-1}}^{a_i}\frac{\varphi^p (t)}{t}\,dt
&\leq&
\int_{a_k}^{(C_1+1)^{i-k}a_k} \frac{\varphi^p (t)}{t}\,dt \\
&\leq&
\varphi^{p-1} ((C_1+1)^{i-k} a_k) \int_{a_k}^{(C_1+1)^{i-k} a_k} \frac{\varphi (t)}{t}\,dt \\
&\leq&
2^{p-1} (C_1+1)^{(p-1)(i-k)} \, \varphi^{p-1} (\frac{a_k}{2}) \int_{a_k}^{(C_1+1)^{i-k} a_k} \frac{\varphi (t)}{t}\,dt \\
&\leq&
2^{p-1} (C_1+1)^{p(i-k)} \, \varphi^{p-1} (\frac{a_k}{2}) \, a_k\frac{\varphi(a_k)}{a_k} \\
&\leq&
2^p (C_1+1)^{p(i-k)} \varphi^{p-1} (\frac{a_k}{2}) \, \int_{a_k/2}^{a_k} \frac{\varphi (t)}{t}\,dt\\
&\leq&
2^p (C_1+1)^{p(i-k)} \, \int_{a_k/2}^{a_k} \frac{\varphi^p(t)}{t}\,dt \\
&\leq&
2^p (C_1+1)^{p(i-k)} \, \int_{a_{k-1}}^{a_k} \frac{\varphi^p(t)}{t}\,dt
\end{eqnarray*}
and so
$$
b^{-pi}\int_{a_{i-1}}^{a_i}\frac{\varphi^p (t)}{t}\,dt \leq 
2^p \Big(\frac{C_1+1}{b} \Big)^{p(i-k)} b^{-pk} \, \int_{a_{k-1}}^{a_k} \frac{\varphi^p(t)}{t}\,dt.
$$
Since $b > C_1 + 1$ we obtain (\ref{estimate21}).

In a similar way, applying (\ref{estimate23}) for $j = k$, we get
\begin{eqnarray*}
b^{-pi}\varphi^p(a_i-a_{i-1})
&\leq& 
b^{-pi}(C_1+1)^{p(i-k)}\varphi^p(a_k)\\
&\leq&
2^p \left(\frac{C_1+1}{b}\right)^{p(i-k)} b^{-pk}\varphi^p(a_k-a_{k-1}),
\end{eqnarray*}
which implies \eqref{estimate22}.

Relations (\ref{estimate20})--(\ref{estimate22}) imply (\ref{LorentzMainIneq}), so we proved the statement for $b > C_1 + 1$. 
To extend this result to all $b > 1$ it suffices to prove the following: whenever (\ref{LorentzMainIneq})
holds for some $b > 1$ and arbitrary non-decreasing sequence 
 $(a_j)_{j = - \infty}^{\infty}$ with a constant $C$, it is automatically fulfilled for $b^{1/2}$ 
 with a constant not exceeding $2^p b^p C$. Indeed, if
 \vspace{-2mm}
$$
y=\sum_{j=-\infty}^{+\infty}b^{-j/2}\chi_{(a_{j-1}, a_j]} ~ \mbox{ and} ~ z=\sum_{j=-\infty}^{+\infty}b^{-j}\chi_{(a_{2j-2}, a_{2j}]},
$$
then
\vspace{-2mm}
\begin{eqnarray*}
\|y\|_{\Lambda_{p,\varphi}}^p
&=&
\sum_{j=-\infty}^{+\infty} b^{-pj/2} \, \int_{a_{j-1}}^{a_j} \varphi^p(t) \frac{dt}{t} \\
&=&
\sum_{j=-\infty}^{+\infty} b^{-p(2j-1)/2} \, \int_{a_{2j-2}}^{a_{2j-1}} \varphi^p(t) \frac{dt}{t} 
+ \sum_{j=-\infty}^{+\infty} b^{-pj} \, \int_{a_{2j-1}}^{a_{2j}} \varphi^p(t) \frac{dt}{t}\\
& \stackrel{b^{p/2}}{\approx} &
\sum_{j=-\infty}^{+\infty} b^{-pj} \, \int_{a_{2j-2}}^{a_{2j}} \varphi^p(t) \frac{dt}{t} = ~
\|z\|_{\Lambda_{p,\varphi}}^p.
\end{eqnarray*}
On the other hand,
$$
b^{-pj} \varphi^p (a_{2j}-a_{2j-2}) ~ \stackrel{2^pb^{p/2}}{\approx}
~ b^{-pj}\varphi^p (a_{2j}-a_{2j-1}) + b^{-p(2j-1)/2}\varphi^p
(a_{2j-1}-a_{2j-2}),
$$
so we get an analog of \eqref{LorentzMainIneq} for $y$ and $b^{1/2}$ and the proof is complete.
\end{proof}

\noindent
\begin{remark} For the space $\Lambda_{1,\varphi}=\Lambda_{\varphi}$ the result can be
proved also by using the following well-known formula {\rm (cf. formula 5.1 in \cite{KPS82} on page 108)}
$$
\|x\|_{\Lambda_{\varphi}} = \sum_{j=-\infty}^{+\infty}(b^{-j}-b^{-j-1}) \varphi
(a_{j}).
$$
\end{remark}

Let, as above, for a given weight $w, M_k=\{t\in[0,1]:w(t)\in[2^k, 2^{k+1})\} ~ (k\in\mathbb Z)$
and $(w_r)_{r=1}^{\infty}$ be the non-increasing rearrangement of the sequence
$(m(M_k))_{k=-\infty}^{+\infty}.$

\begin{theorem}\label{PropLorentzDecompCrit}
Let $\varphi$ be an increasing concave function on $[0,1]$ such that $\gamma_\varphi>0$,
$1\le p<\infty$ and let $w$ be a weight on $[0,1]$. Then the Lorentz space $X:=\Lambda_{p,\varphi}$
is $w$-decomposable if and only if $\varphi$ satisfies condition \eqref{condOnFundFunc}.
\end{theorem}

\begin{proof}
If $X = \Lambda_{p,\varphi}$ is $w$-decomposable then, by Corollary \ref{cor1new}, the relation 
\eqref{condOnFundFunc} holds for the fundamental function $\varphi_X.$ Since, as
it was mentioned above, $\varphi \approx \varphi_X$, then \eqref{condOnFundFunc} is
fulfilled for $\varphi$ as well.

Conversely, suppose that $\varphi$ satisfies \eqref{condOnFundFunc}.
Let $n\in\mathbb N$ and $x_1, x_2, \dots, x_n$ be non-negative functions from $X$ satisfying
\eqref{inclForSupp}. Evidently, there exist $x_1', x_2', \dots, x_n' \in X$
taking their values from the set $\{2^{-k}\}_{k=-\infty}^{\infty} \cup \{0\}$
and such that $x_i(t) \stackrel{2}{\approx} x_i'(t)$ $(0<t\le 1).$
Clearly, $m({\rm supp}\, x_i') = m({\rm supp}\, x_i) \le w_i$ $(1\le i\le n)$ and
$$
m \Big\{t:\sum_{i=1}^n
x_i'(t) = 2^{-k}\Big\} = \sum_{i=1}^n m\{t:x_i'(t)=2^{-k}\}
$$
for all integer $k.$ Therefore, applying \eqref{condOnFundFunc}, we get that
\begin{equation}\label{ApplCondOnFundFunc}
\sum_{i=1}^n\varphi^p(m\{t:x_i'(t)=2^{-k}\}) ~ {\approx} ~ \varphi^p(m\{t:\sum_{i=1}^n
x_i'(t)=2^{-k}\})\;\;(k\in\mathbb{Z}).
\end{equation}
On the other hand, Proposition \ref{LorentzMainProp} yields
\vspace{-2mm}
\begin{equation}\label{decompOfxi}
\|x_i'\|_{X}^p ~ {\approx} ~ \sum_{k=-\infty}^{+\infty} 2^{-pk}\varphi^p (m\{t:x_i'(t)=2^{-k}\})\;\;(1\le i\le n)
\end{equation}
and
\begin{equation}\label{decompOfSumxi}
\|\sum_{i=1}^n x_i'\|_{X}^p ~ {\approx} ~ \sum_{k=-\infty}^{+\infty}2^{-pk} \varphi^p
(m\{t:\sum_{i=1}^n x_i'(t)=2^{-k}\})
\end{equation}
with a constant which depends only on $\varphi$ and $p.$ Combining
relations \eqref{decompOfxi} and \eqref{decompOfSumxi} with
\eqref{ApplCondOnFundFunc}, we obtain \eqref{decompForSym7} for
$x_i'$ and so for $x_i.$ The proof is complete.
\end{proof}

In particular, from the above theorem and a remark after Theorem \ref{ExponentM}
it follows that the Lorentz space $\Lambda_{\varphi}$ generated by the function 
$\varphi(t)=t\,\log\frac{e}{t}$ is $1/t$-decomposable and therefore the Banach couple 
$(\Lambda(\varphi), \Lambda(\varphi)(\frac{1}{t}))$ is $K$-monotone.

\begin{theorem}\label{Cor on non-tri. weights}
Suppose that $\varphi$ is an increasing  concave function on $[0,1]$ such
that $\gamma_\varphi>0$ and $1\le p<\infty.$ The following
conditions are equivalent:

(a) there exists a weight $w$ on $[0,1]$ such that the Lorentz space
$\Lambda_{p,\varphi}$ is $w$-decomposable;

(b) $\varphi$ is equivalent to a regularly varying function at zero
of order $p$.
\end{theorem}

\begin{proof}
First, if $X:=\Lambda_{p,\varphi}$ is $w$-decomposable for some
weight $w$ on $[0, 1]$, then, by Corollary \ref{cor2}, as in the
proof of the previous theorem, we conclude that $\varphi$ is
equivalent to a regularly varying function at zero of order $p.$

Conversely, suppose that $\varphi$ is equivalent to a function that
varies regularly at zero with order $p,$ that is, $\varphi$ satisfies
\eqref{regVarEquivalent} for some $\tau(N)$ $(N=1,2,\dots)$. Consider a family
$(M_N)_{N=1}^{\infty}$ of pairwise disjoint measurable subsets of
$[0, 1]$ such that $m(M_2)=\min(\tau(2),1/4),$
$$
m(M_N)=\min(\tau(N),\frac{m(M_{N-1})}{2}),\;\;N>2,
$$
and let $M_1:=[0, 1]\setminus\bigcup_{N=2}^{\infty}M_N.$ Set
$w(t):=2^N$ for all $t\in M_N$ and $N\in\mathbb N.$ Clearly, $m
(M_{N+1})\le m(M_N)/2$ $(N\in\mathbb N).$ Therefore, by Theorem
\ref{ExponentM}, $\varphi$ satisfies \eqref{condOnFundFunc} for any
sequence $(\tau_N)_{N=1}^{\infty}$ majorized by the sequence $(m
(M_N))_{N=1}^{\infty}$. To complete the proof it remains to apply Theorem
\ref{PropLorentzDecompCrit}.
\end{proof}

It is obvious that $L_p$-spaces $(1\le p\le\infty)$ are $w$-decomposable for
every weight $w.$ On the other hand, we show that for an arbitrary weight $w$ 
there exist $w$-decomposable Lorentz spaces $\Lambda_{\varphi}$ different 
from $L_1$.

\begin{theorem}\label{nontrivialspaceexists}
Let $w$ be an arbitrary weight on $[0, 1]$. Then there exists an increasing concave
function $\varphi$ such that the space $\Lambda_{\varphi}$ is $w$-decomposable and 
$\Lambda_{\varphi} \ne L_1$.
\end{theorem}

\begin{proof}
As above, $M_k=\{t\in[0,1]: w(t)\in[2^k, 2^{k+1})\}$ for $k\in\mathbb Z$
and $(w_r)_{r=1}^{\infty}$ \konst{is} the non-increasing rearrangement of the sequence
$(m(M_k))_{k=-\infty}^{+\infty}$. Define
$$
G(\alpha):=\sum\limits_{r=1}^{\infty}\min\{\alpha,w_r\},\;\;\alpha\ge 0.
$$

\vspace{-2mm}

\noindent
Then $G(1)=1, G(0)=0$ and $G$ is increasing and continuous at
zero. 

Let $(t_k)_{k=0}^{\infty}$ be a sequence from $(0, 1]$ such that $t_0 = 1, 0 < t_k < t_{k-1}/3$ 
for $k \geq 1$ and
\begin{equation} \label{inequality27}
G(t_{k+1}) \leq 2^{-k} \, t_k, ~ k = 0, 1, \ldots.
\end{equation}
Then we set
$\varphi_k'(t) = \max_{i = 0, 1, \ldots, k} \{2^i \chi_{[0, t_i]}(t)\}, k = 0, 1, \ldots$ and 
$\varphi'(t) = \lim_{k \rightarrow \infty} \varphi_k'(t) ~ (0 < t \leq 1)$. It is easy to see 
that $\varphi_k'$ and $\varphi'$ are non-increasing functions on$(0, 1]$. Moreover,
since 
$$
t_k \varphi'(t_k) = t_k \, 2^k \leq \frac{2}{3} t_{k-1}\,2^{k-1} =  \frac{2}{3} t_{k-1} \varphi'(t_{k-1})
$$
it follows that 
$$
\int_0^1 \varphi'(t)dt \leq \sum_{k=0}^{\infty} \varphi'(t_k) t_k \leq  \sum_{k=0}^{\infty} \left(\frac{2}{3}\right)^k <\infty.
$$
Therefore, the function $\varphi(t): = \int_0^t \varphi'(s)\,ds$ is well-defined, increasing and concave on $(0, 1]$. 
We shall prove that the Lorentz space $\Lambda_{\varphi}$ is $w$-decomposable.

In view of Theorem \ref{PropLorentzDecompCrit}, it suffices to show that for some constant $C\ge 1$ and
for any sequence of reals $(d_r)_{r=1}^{\infty}$ such that $0<d_r\le w_r$ $(r=1, 2, \dots)$ we
have
$$
\varphi\left(\sum_{r=1}^\infty d_r\right)
\le\sum_{r=1}^\infty\varphi(d_r)\le C\varphi \left(\sum_{r=1}^\infty d_r\right).
$$
Note that the left hand side of this inequality is an
immediate consequence of the concavity of $\varphi.$ Further, since
$\varphi_k(t):=\int_0^t \varphi_k'(s)\,ds \uparrow \varphi(t),$ then
$\lim_{k\to\infty}\sum_{r=1}^\infty\varphi_k(d_r)=\sum_{r=1}^\infty\varphi(d_r).$
Therefore, it is enough to prove that
\begin{equation}\label{FastIncW_IneqToProve}
\frac{\sum_{r=1}^\infty\varphi_k(d_r)}{\varphi_k\left(\sum_{r=1}^\infty
d_r\right)}\le 3,\;\;k\ge 0.
\end{equation}
Noting that $\sum_{r=1}^\infty d_r\le t_0=1,$ we set
$$
k_0:=\max\Big\{k= 0, 1, 2, \dots:\, \sum\limits_{r=1}^\infty d_r\le t_k\Big\}.
$$
From the definition of $\varphi_k$ it follows that
\begin{equation}\label{simple case}
\varphi_k\left(\sum_{r=1}^\infty d_r\right) = 2^k \, \sum_{r = 1}^{\infty} d_r 
= \sum_{r=1}^\infty\varphi_k(d_r)\;\;\mbox{if}\;\;0\le k \leq k_0.
\end{equation}
Since $t_{k_0+1} < \sum\limits_{r=1}^\infty d_r \leq t_{k_0},$ then,
again by the definition of $\varphi_k$,
\begin{equation}\label{middle case}
\sum_{r=1}^\infty\varphi_{k_0+1}(d_r)\le 2^{k_0+1}\sum_{r=1}^\infty
d_r\le 2\varphi_{k_0+1}\left(\sum_{r=1}^\infty d_r\right).
\end{equation}
Let $k>k_0$ be arbitrary. The inequality $\sum\limits_{r=1}^\infty d_r> t_k$ implies 
that 
\begin{equation} \label{inequality31}
\varphi_k\left(\sum_{r=1}^\infty d_r\right) > \varphi_k(t_k) = 2^k t_k.
\end{equation}
Moreover, since
$$
\varphi_{k+1}(d_r) =
\begin{cases}
2^{k+1} d_r = 2\, \varphi_k(d_r), & {\rm if} ~ d_r \leq t_{k+1},\\
2^k t_{k+1} + \varphi_k(d_r), & {\rm if} ~ d_r > t_{k+1},
\end{cases}
$$
we obtain
$$
\sum_{r=1}^\infty\varphi_{k+1}(d_r) -
\sum_{r=1}^\infty\varphi_k(d_r)= \sum_{r=1}^\infty
\min(2^kt_{k+1},2^k d_r)\le 2^k G(t_{k+1}).
$$
Hence, for any $k > k_0$, by (\ref{inequality31}) and (\ref{inequality27}), we obtain
\begin{eqnarray*}
\frac{\sum_{r=1}^\infty\varphi_{k+1}(d_r)}{\varphi_{k+1}\left(\sum_{r=1}^\infty
d_r\right)}&\le&\frac{\sum_{r=1}^\infty\varphi_k(d_r)}{\varphi_k\left(\sum_{r=1}^\infty
d_r\right)}+\frac{\sum_{r=1}^\infty\varphi_{k+1}(d_r)-\sum_{r=1}^\infty\varphi_k(d_r)}{\varphi_k\left(\sum_{r=1}^\infty
d_r\right)}\\&\le&
\frac{\sum_{r=1}^\infty\varphi_k(d_r)}{\varphi_k\left(\sum_{r=1}^\infty
d_r\right)}+\frac{G(t_{k+1})}{t_k}\le\frac{\sum_{r=1}^\infty\varphi_k(d_r)}{\varphi_k\left(\sum_{r=1}^\infty
d_r\right)}+2^{-k}.
\end{eqnarray*}
Applying the last estimate together with \eqref{simple case} and \eqref{middle case},
we obtain \eqref{FastIncW_IneqToProve}. It is easy to see that $\varphi(t)$ is
not equivalent to $t$, and therefore $\Lambda_{\varphi}\ne L_1$. The proof is complete.
\end{proof}

\begin{remark} Theorem \ref{nontrivialspaceexists} can be easily
extended to the spaces $\Lambda_{p,\psi}$ with $p \in (1, \infty)$.
Indeed, let $w$ be an arbitrary weight on $[0, 1]$ and $\varphi$ be
the function from the proof of Theorem \ref{nontrivialspaceexists}.
Set $\psi:=\varphi^{1/p}.$ Clearly, $\psi$ is an increasing concave
function not equivalent to the function $t^{1/p}.$ Therefore,
$\Lambda_{p,\psi}\ne L_p.$ Since relation \eqref{condOnFundFunc} is
fulfilled for $\psi$ as well, then, by Theorem
\ref{PropLorentzDecompCrit}, the space $\Lambda_{p,\psi}$ is
$w$-decomposable.
\end{remark}
\vspace{2mm}

Our next goal is to prove analogous results for Marcinkiewicz spaces
$M_{\varphi}.$ To make use of the duality of Lorentz and Marcinkiewicz
spaces we will need the following statement which is of interest in its own right.

\begin{theorem}\label{Theorem4.5}
Let $X$ be a Banach lattice on a $\sigma$-finite measure space $(\Omega, \Sigma, \mu)$ 
with ${\rm supp} X = \Omega$ which has the Fatou property and $w$ be a non-trivial
weight on $\Omega$. Then the couple $(X, X(w))$ is $K$-monotone if and only if
$(X^{\prime}, X^{\prime}(w))$ is $K$-monotone, where $X^{\prime}$ is the K\"othe dual of $X$.
\end{theorem}
\vspace{2mm}

The proof follows from Theorem \ref{Theor Tikhomirov} proved in \cite{Ti11} and the following 
result.
\vspace{2mm}

\begin{theorem}\label{PropForDual}
Let $X$ be a Banach lattice on a $\sigma$-finite measure space $(\Omega, \Sigma, \mu)$ 
with ${\rm supp} X = \Omega$ which has the Fatou property and $w$ be a non-trivial
weight on $\Omega$. Then $X$ is $w$-decomposable if and only if its K\"othe dual 
$X^{\prime}$ is $w$-decomposable.
\end{theorem}

\begin{proof}
Suppose that $X$ is $w$-decomposable. Let $n\in\mathbb N$ and the functions
$x_1', x_2', \dots, x_n'$, $y_1', y_2', \dots, y_n'\in X'$ satisfy \eqref{equalnorms}
(with the norm from $X'$) and \eqref{arrangement}. Take a function
$x\in X,$ $\|x\|_X=1,$ such that ${\rm supp} \, x\subset\bigcup\limits_{i=1}^n {\rm supp}\, x_i'$ and
$$
\|\sum_{i=1}^n x_i' \|_{X'}\le 2\int_{\Omega} | \sum_{i=1}^n x_i'(t) x(t) |\,d\mu.
$$
Now, consider
$y_i\in X$ such that ${\rm supp}\, y_i\subset {\rm supp}\, y_i',$
$\|y_i\|_X=\|x\chi_{{\rm supp}\, x_i'}\|_X$ and
$$
\|y_i'\|_{X'}\le\frac{2}{\|y_i\|_X}\int_{\Omega} |y_i'(t) y_i(t)|\, d\mu,\;\;1\le i\le n.
$$
Then, according to the hypothesis,
$$
\|\sum_{i=1}^n y_i \|_X\le C \|\sum_{i=1}^n x\chi_{{\rm supp}\, x_i'}\|_X = C,
$$
and, therefore,
\begin{eqnarray*}
\|\sum_{i=1}^n y_i' \|_{X'}
&\ge&
\frac{1}{C} \int_{\Omega} |\sum_{i=1}^n y_i(t)\sum_{j=1}^n
y_j'(t) |\,d\mu=\frac{1}{C}\sum\limits_{i=1}^n \int\limits_{\Omega}
|y_i'(t)y_i(t)|\,d\mu\\ &\ge& \frac{1}{2C}\sum\limits_{i=1}^n
\|y_i'\|_{X'}\|y_i\|_X=\frac{1}{2C}\sum\limits_{i=1}^n
\|x_i'\|_{X'}\|x\chi_{{\rm supp}\, x_i'}\|_X\\
&\ge&
\frac{1}{2C}\sum\limits_{i=1}^n\int_{\Omega}
|x_i'(t)x(t)\chi_{{\rm supp}\, x_i'}(t)|\,d\mu\ge\frac{1}{4C}
\|\sum\limits_{i=1}^n x_i' \|_{X'}.
\end{eqnarray*}
Certainly, the same argument can be applied to get the opposite
estimate. The proof is complete.
\end{proof}

Since $M_{\varphi}^{\prime} = \Lambda_{\tilde{\varphi}}$ (cf. \cite{KPS82}, p. 117) and 
$\delta_{\varphi} + \gamma_{\tilde{\varphi}} = 1$ for any increasing concave function 
$\varphi$ on $[0, 1]$ (cf. \cite{KPS82}, Theorem 4.12 on page 107 or \cite{Ma85}, p. 28), then by 
Theorems \ref{PropLorentzDecompCrit}, \ref{Cor on non-tri. weights} and \ref{PropForDual} 
we immediately obtain the following statements.

\begin{corollary}\label{cor3}
Let $\varphi$ be an increasing concave function on $[0,1]$ such that
$\delta_\varphi< 1$ and let $w$ be a weight on $[0, 1]$. Then the
Marcinkiewicz space $M_\varphi$ is $w$-decomposable if and only if
$\tilde{\varphi}(t) =t/\varphi(t)$ satisfies \eqref{condOnFundFunc}
with $p=1.$
\end{corollary}

\begin{corollary}\label{cor4}
If $\varphi$ is an increasing concave function on $[0,1]$ such that
$\delta_\varphi<1$, then the space $M_\varphi$ is $w$-decomposable
for some weight $w$ on $[0, 1]$ if and only if $\varphi$ is
equivalent to a regularly varying function at zero of order
$\infty$.
\end{corollary}

In the paper \cite{Ka93}, Kalton proved that if $X$ and $Y$ are
symmetric sequence spaces with the Fatou property such that the
couple $(X, Y(w))$ is $K$-monotone for some non-trivial weight $w$,
then $X= l_p$ and $Y = l_q$ with $1 \leq p, q \le \infty$. The results 
in this section and Theorem  \ref{Theor Tikhomirov} show that in the case of
symmetric function spaces on $[0, 1]$ the situation is completely
different. The following theorems present new examples of
$K$-monotone Banach couples of weighted Lorentz and Marcinkiewicz
function spaces. The first theorem follows from Theorem  \ref{Theor Tikhomirov}, 
Theorem \ref{Cor on non-tri. weights}, Theorem \ref{nontrivialspaceexists}
and Remark 2 and the second one from Theorem  \ref{Theor Tikhomirov}, 
Theorem \ref{PropForDual} on the duality and Corollary \ref{cor4}.

\begin{theorem}\label{K-monotoneLorentz}
If $\varphi$ is an increasing  concave function on $[0,1]$ such that
$\gamma_\varphi>0$ and $1\le p<\infty$, then the weighted couple
$(\Lambda_{p,\varphi}, \Lambda_{p,\varphi}(w))$ is $K$-monotone for
some (non-trivial) weight $w$ on $[0,1]$ if and only if $\varphi$ is
equivalent to a regularly varying function at zero of order $p$. On
the other hand, for arbitrary weight $w$ on $[0, 1]$ and $1 \leq p <
\infty$ there exists an increasing concave function $\varphi$ on
$[0, 1]$ such that the couple $(\Lambda_{p,\varphi},
\Lambda_{p,\varphi}(w))$ is $K$-monotone and $\Lambda_{p,\varphi}
\neq L_p$.
\end{theorem}

\begin{theorem}\label{K-monotoneMarcinkiewicz}
If $\varphi$ is an increasing  concave function on $[0,1]$ such that
$\delta_{\varphi} < 1$, then the weighted couple $(M_{\varphi},
M_{\varphi}(w))$ is $K$-monotone for some (non-trivial) weight $w$
on $[0,1]$ if and only if $\varphi$ is equivalent to a regularly
varying function at zero of order $\infty$.
\end{theorem}

\begin{section}
{\bf $w$--decomposable Orlicz spaces}
\end{section}

As we have seen in the previous section, in order to check the property of $w$-decomposa\-bility 
for Lorentz spaces, it is enough to consider only characteristic functions (Theorem 
\ref{PropLorentzDecompCrit}). In this section we will prove that in the case of Orlicz spaces it is 
sufficient to examine scalar multiples of characteristic functions.

As above, for a weight $w$ on $[0,1]$ let
$M_k:=\{t\in[0,1]:w(t)\in[2^k, 2^{k+1})\}$ $(k\in\mathbb Z),
(w_r)_{r=1}^{\infty}$ be the non-increasing rearrangement of the
sequence $(m(M_k))_{k=-\infty}^{+\infty}$ and $\{\bar
{M_r}\}_{r=1}^{\infty}$ denote any rearrangement of the sets $M_k$
such that $m(\bar{M_r}) = w_r, r = 1, 2, \ldots$
\vspace{2mm}

\begin{theorem}\label{interestinglemma}
Let an Orlicz function $F$ satisfy the $\Delta_{2}$-condition for
large $u$ and let $w$ be a weight on $[0, 1]$. Then, the Orlicz
space $L_F = L_F[0, 1]$ is $w$-decomposable if and only if there
exists $p\in [1,\infty)$ such that for any $n\in\mathbb N,$ all
measurable sets $A_k\subset \overline {M}_k$ and reals $c_k$ $(1\le k \le
n)$ we have
\begin{equation}\label{suffcondfororlicz}
\|\sum_{k=1}^n c_k\chi_{A_k} \|_{L_F}^p~ {\approx} ~ \sum_{k=1}^n\|
c_k\chi_{A_k}\|_{L_F}^p
\end{equation}
with a constant independent of $c_k, A_k ~ (1 \leq k \leq n)$ and $n \in \mathbb N$.
If, in addition, the complementary function $F^*$ satisfies the $\Delta_2$-condition
for large $u$, then the $w$-decomposability of $L_F$ implies that $F$ is equivalent to a regularly 
varying Orlicz function at $\infty$ of order $p$.
\end{theorem}

\begin{proof}
Suppose, first, that $L_F$ is $w$-decomposable. By Proposition \ref{prop2}, there is $p\in [1,\infty]$
such that \eqref{decompForSym7} holds for $X=L_F,$ which implies
\eqref{suffcondfororlicz}. Since $F$ satisfies the
$\Delta_{2}$-condition for large $u>0$, then $\alpha_X>0$. Therefore,
by Corollary \ref{cor2}, $p<\infty.$
\vspace{2mm}

Conversely, let $n\in\mathbb N$ and $y_k\in L_F,$ ${\rm supp}\, y_k\subset \bar {M_k}$,
$1\le k\le n.$ We may (and will) assume that $y_k$ are positive bounded functions
and
\begin{equation}\label{sum of norms}
\sum_{k=1}^n \|y_k\|^p_{L_F}=1.
\end{equation}
Taking into account Theorem \ref{thm 1}, we need to show that
\begin{equation}\label{norm of sum}
\|\sum_{k=1}^n y_k \|^p_{L_F} ~ {\approx} ~ 1,
\end{equation}
with a constant independent from $n$ and $y_k.$ For each $1\le k\le
n$ we set
$$
c_k = \frac{\|y_k\|_{L_F}}{2\varphi_{L_F}(m({\rm supp}\, y_k))}
$$
and
$$
\tilde y_k(t):=
\begin{cases}
y_k(t),& {\rm if} ~ y_k(t) \geq c_k,\\
0,& {\rm if} ~y_k(t) < c_k.
\end{cases}
$$
Applying (\ref{suffcondfororlicz}) to the functions $c_k \chi_{{\rm supp} y_k}$ and taking into 
account the definition of $c_k$ and (\ref{sum of norms}) we get
\begin{eqnarray*}
\| \sum_{k=1}^n c_k  \chi_{{\rm supp} y_k} \|_{L_F}^p 
&\leq&
C_1  \, \sum_{k=1}^n c_k^p  \varphi_{L_F}(m({\rm supp} y_k))^p\\
&=&
C_1 \, \sum_{k=1}^n 2^{-p} \, \| y_k \|_{L_F}^p = 2^{-p} C_1.
\end{eqnarray*}
Up to equivalence of norms the Orlicz space $L_F=L_F[0, 1]$ depends only on the behaviour of $F$ 
for large enough $u > 0$. Therefore, we may assume that $F( 2u ) \leq C_2 F(u)$ for all $u > 0$. 
Then, from the last inequality it follows that
\begin{equation} \label{inequality35}
\sum_{k=1}^n m({\rm supp} y_k) F(c_k) \leq C_3,
\end{equation}
where $C_3$ is a constant independent of $n$ and $y_k$. Moreover, from the definition of $c_k$ 
and $\tilde y_k$ we have
\begin{equation} \label{inequality36}
\|\tilde y_k\|_{L_F}\le\|y_k\|_{L_F} ~~ {\rm and} ~~ \|\tilde y_k\|_{L_F} \geq
\|y_k\|_{L_F}-\|c_k \chi_{{\rm supp}\, y_k}\|_{L_F}=\frac{1}{2}\|y_k\|_{L_F}.
\end{equation}
Next, let us show that there is $r_k \in [c_k, \sup_{t}\tilde y_k(t)]$ such that 
\begin{equation}\label{equationforri}
F(r_k) = F\left(\frac{r_k}{\|\tilde y_k\|_{L_F}}\right)\int_0^1 F(\tilde y_k(t))dt.
\end{equation}
In fact, consider the function
$$
H_k(t): = \frac{F(\tilde y_k(t))}{F\left(\frac{\tilde y_k(t)}{\|\tilde y_k(t)\|_{L_F}}\right)}, ~ t \in {\rm supp} \, \tilde y_k.
$$
From the equality $\int_0^1 F(\frac{\tilde y_k(t)}{\|\tilde y_k(t)\|_{L_F}})dt=1$ it follows that
$$
\inf_{t\in {\rm supp}\, \tilde y_k} H_k(t) \le \int_0^1 F[\tilde y_k(t)]dt \leq \sup\limits_{t\in {\rm supp}\, \tilde y_k}H_k(t).
$$
Thus, since $\inf\limits_{t\in{\rm supp}\, \tilde y_k} \tilde y_k(t) \geq c_k$, by the continuity of $F$, equality \eqref{equationforri} 
holds for some $r_k$ from the interval $ [c_k, \sup_{t}\tilde y_k(t)]$.

Next, define $d_k\in [0, 1]$ $(k = 1,2, \ldots, n)$ as follows:
$$
d_k =
\begin{cases}
\varphi_{L_F}^{-1}\left(\frac{\|\tilde y_k\|_{L_F}}{r_k}\right),& {\rm if}�~ \|\tilde y_k\|_{L_F} \leq r_k \varphi_{L_F}(m({\rm supp}\, y_k)),\\
m({\rm supp}\, y_k),& {\rm if} ~ \| \tilde y_k\|_{L_F} > r_k \varphi_{L_F}(m({\rm supp}\, y_k)).
\end{cases}
$$
Clearly, by the definition of $d_k$,
\begin{equation}\label{riphidi2}
r_k\varphi_{L_F} (d_k) \leq \|\tilde y_k\|_{L_F}.
\end{equation}
On the other hand, since $r_k \geq c_k,$ we obtain
\begin{equation}\label{riphidi}
r_k \varphi_{L_F}(d_k) \geq \frac{1}{2} \|\tilde y_k\|_{L_F},
\end{equation}
whence $d_k \geq \varphi_{L_F}^{-1}( \|\tilde y_k\|_{L_F} /(2 r_k)) $. Hence, taking into account that 
$F$ satisfies the $\Delta_2$-condition with constant $C_2$ for all $u > 0$, the formula 
$ \varphi_{L_F}(t) = 1/F^{-1}(1/t)$ (see formula (9.23) in \cite{KR61} on page 79 of the English version or 
Corollary 5 in \cite{Ma89} on page 58) and (\ref{equationforri}), we have
\begin{equation}\label{riphidi3}
d_k F(r_k) \geq \frac{F(r_k)}{F\left(\frac{2r_k}{\|\tilde y_k\|_{L_F}}\right)}\ge \frac{1}{C_2}
\frac{F(r_k)}{M\left(\frac{r_k}{\|\tilde y_k\|_{L_F}}\right)} = \frac{1}{C_2} \int\limits_0^1 F[\tilde y_k(t)]\,dt.
\end{equation}
Conversely, from the equality $1/d_k = F\left(1/\varphi_{L_F}(d_k)\right)$, (\ref{riphidi2}) and
\eqref{equationforri} it follows that
\begin{equation}\label{estimate41}
d_k F(r_k) = \frac{F(r_k)}{F(\frac{1}{\varphi_{L_F}(d_k)})} \leq \frac{F(r_k)}{F\left(\frac{r_k}{\|\tilde y_k\|_{L_F}}\right)} 
= \int\limits_0^1 F[\tilde y_k(t)]\,dt.
\end{equation}
Now, by the definition of $d_k$, we have $d_k \leq m({\rm supp} \, y_k)$. Therefore, we can define the scalar 
multiples of characteristic functions $f_k(t) := r_k \chi_{B_k}(t),$ where $B_k \subset {\rm supp}\, y_k$ and
$m(B_k) = d_k.$ 
According to \eqref{riphidi2}, \eqref{riphidi} and (\ref{inequality36}), we have
$$
\frac{1}{4}\|y_k\|_{L_F} \leq \|f_k\|_{L_F} \leq \|y_k\|_{L_F}, k = 1, 2, \ldots, n.
$$
Therefore, in view of \eqref{suffcondfororlicz} and \eqref{sum of
norms}, we obtain
$$
\|\sum\limits_{k=1}^n f_k \|_{L_F}^p ~ {\approx} ~ \sum_{k=1}^n\|f_k\|_{L_F}^p
~ \approx ~ \sum_{k=1}^n\|y_k\|_{L_F}^p=1,
$$
with constants which depend only on $p$. Hence, taking into account that $F$ satisfies the 
$\Delta_{2}$-condition, we conclude that \eqref{norm of sum} will be proved once we show that
\begin{equation*}
\| \sum_{k=1}^n y_k \|_{L_F} ~ \approx ~ \| \sum_{k=1}^n f_k \|_{L_F}
\end{equation*}
with constants independent of $n$ and $y_k$. Since the functions $f_k$ (respectively, $y_k$) are pairwise disjoint, in view of 
estimate \eqref{estimate41}, we find that
\begin{eqnarray*}
\int_0^1 F[\sum_{k=1}^n f_k(t)]\,dt
&=&
\sum_{k=1}^n d_k F(r_k) \leq \sum_{k=1}^n \int_0^1 F(\tilde y_k(t))\,dt \\
& \leq&
\int_0^1 F[\sum_{k=1}^n y_k(t)]\,dt.
\end{eqnarray*}
Conversely, by \eqref{riphidi3} and (\ref{inequality35}), we get
\begin{eqnarray*}
\int_0^1 F[\sum_{k=1}^n y_k(t)]\,dt
&\le&
\sum_{k=1}^n\int_0^1 F[\tilde y_k(t)]\,dt+\sum_{k=1}^n m({\rm supp}\, y_k) F(c_k)\nonumber\\
&\leq&
C_2\int_0^1 F[\sum_{k=1}^n f_k(t)]\,dt + C_3,
\end{eqnarray*}
and we come to the desired result.
\vspace{1mm}

In order to obtain the second assertion of the theorem it is sufficient to apply Corollary \ref{cor1new}, 
Lemmas \ref{lemma1} and \ref{lemma2}, Proposition \ref{equiv of functions} and the elementary 
observation that condition (a) in that proposition implies the equivalence of $F$ to an Orlicz function 
which is regularly varying at $\infty$ of order $p$.
\end{proof}

\noindent
\begin{remark} {\rm Arguing in the same way as in the proof of Theorem \ref{interestinglemma} we may obtain 
the following result:} Let an Orlicz function $F$ satisfy the $\Delta_{2}$-condition for large $u$ and $1<p, q <\infty.$
The Orlicz space $L_F[0, 1]$ satisfies the upper $p$-estimate, respectively the lower $q$-estimate, if and
only if there exsists a constant $C>0$ such that for any $n\in\mathbb N,$
all pairwise disjoint measurable sets $A_k$ and reals $c_k$ we have
\begin{equation*}
\|\sum_{k=1}^n c_k\chi_{A_k} \|_{L_F} \le C
( \sum_{k=1}^n\| c_k \chi_{A_k}\|_{L_F}^p )^{1/p},
\end{equation*}
respectively,
\begin{equation*}
(\sum_{k=1}^n \| c_k \chi_{A_k} \|_{L_F}^q)^{1/q} \leq C \|  \sum_{k=1}^n c_k \chi_{A_k}\|_{L_F}.
\end{equation*}
\end{remark}

However, an inspection of the proof of results from \cite{KMP97} (pages 120-121 and 124) shows that the first 
of these inequalities is equivalent to either of the following conditions: the Orlicz space $L_F[0, 1]$ is $p$-convex 
or $L_F[0, 1]$ satisfies the upper $p$-estimate or there exists an Orlicz function $F_1$ equivalent to $F$ for large 
arguments such that $F_1(u^{1/p})$ is a convex function on $[0, \infty)$. At the same time, the second of them is 
equivalent to either of the following conditions: the Orlicz space $L_F[0, 1]$ is $q$-concave or $L_F[0, 1]$ satisfies 
the lower $q$-estimate or there exists an Orlicz function $F_1$ equivalent to $F$ for large arguments such that 
$F_1(u^{1/q})$ is a concave function on $[0, \infty)$.
\vspace{2mm}

The following result is analogous to Theorem \ref{Cor on non-tri. weights} for Lorentz spaces.

\begin{theorem}\label{regularly var and Orlicz}
Let $F$ be an Orlicz function equivalent to an Orlicz function which
is regularly varying at $\infty$ of order $p\in [1,\infty)$.
Then there is a weight $w$ on $[0, 1]$ such that the Orlicz space
$L_F$ is $w$-decomposable and, consequently, the couple $(L_F,
L_F(w))$ is $K$-monotone.
\end{theorem}

\begin{proof}
By Corollary \ref{cor1}, it is sufficient to find a sequence of pairwise disjoint intervals
$\{\Delta_k\}_{k=1}^\infty$ from $[0,1]$ such that for any $n\in\mathbb N$
and $x_1, x_2, \dots, x_n \in X$ satisfying the condition
${\rm supp}\, x_i\subset \Delta_i$ $(1\le i\le n)$, relation \eqref{decompForSym7} holds.

First, since $F$ is equivalent to a regularly varying Orlicz function at $\infty$
of order $p$, then Lemma 1 and a simple compactness argument (see also
\cite[Lemma 6.1]{Ka93}) show that there exists a constant $C_1 > 1$ such that
for every $k\in\mathbb N$ there is $v_k>0$ such that for all $v\ge v_k$ and
$u\in [k^{-2}/8,1]$ we have that
\begin{equation}\label{reg varying}
F(uv) ~ \stackrel{C_1}{\approx} ~ u^p F(v).
\end{equation}

Let $v>0$, $\varepsilon>0$ be arbitrary and $\Delta$ be an interval
from $[0,1]$ such that $m(\Delta)\le \varepsilon/F(v).$ Moreover,
suppose that $z\in L_F,$ $z\ge 0$ and ${\rm supp}\,z \subset\Delta$.
Then
$$
\int_{\{t\in\Delta:\,z(t)\le v\}} F[z(t)]\,dt\le F(v) \, m(\Delta)\leq \varepsilon.
$$
Let $\{\Delta_k\}_{k=1}^\infty$ be a sequence of
disjoint intervals from $[0,1]$ such that
\begin{equation*}
m(\Delta_k)\le 2^{-k-1}(F(v_k))^{-1}\;\;(k=1,2,\dots).
\end{equation*}
Then, as it was noted above, for every $z\in L_F$ such that $z\ge
0$ and ${\rm supp}\,z \subset\Delta_k$, we have
\begin{equation}\label{inegral ineq}
\int_{\{t\in\Delta_k:\,z(t)\le v_k\}} F[z(t)]\,dt\le 2^{-k-1}\;\;(k=1,2,\dots).
\end{equation}
Suppose that $\{x_k\}_{k=1}^\infty$ is an arbitrary sequence from $L_F$
such that $x_k\ge 0$ and ${\rm supp}\,x_k\subset\Delta_k$
$(k=1,2,\dots).$ To prove \eqref{decompForSym7} we assume that
$$
\Big\|\sum_{i=1}^n x_i\Big\|_{L_F}=1,
$$
or, equivalently,
\begin{equation}\label{equality for norm}
\sum_{i=1}^n\int_{\Delta_i} F[x_i(t)]\,dt=1.
\end{equation}
If $\lambda_i:=\|x_i\|_{L_F}$ $(i=1,2,\dots),$ then
$0\le\lambda_i\le 1$ and
\begin{equation}\label{norms of summands}
\int_{\Delta_i} F\Big[\frac{x_i(t)}{\lambda_i}\Big]\,dt=1\;\;
(i=1,2,\dots).
\end{equation}
Denote $I_1:=\{i=1,2,\dots,n:\,\lambda_i\le i^{-2}/8\},$
$I_2:=\{1,2,\dots,n\}\setminus I_1.$ Then
\begin{equation}\label{ineq for I_1}
\sum_{i\in I_1}\lambda_i^p\le \frac18\sum_{i\in I_1} i^{-2p}\le
\frac14.
\end{equation}
Now, let $i\in I_2,$ i.e., $\lambda_i\ge i^{-2}/8.$ Then, if $x_i(t) \ge
\lambda_iv_i$, from \eqref{reg varying} it follows that
\begin{equation}\label{main ineq}
C_1^{-1}\lambda_i^p F\Big[ \frac{x_i(t)}{\lambda_i}\Big ] \le
F[x_i(t)]\le C_1\lambda_i^p F\Big[ \frac{x_i(t)}{\lambda_i}\Big].
\end{equation}
Moreover, by \eqref{inegral ineq} and \eqref{norms of summands}, we have 
\begin{eqnarray*}
\int_{\{t\in\Delta_i:\,x_i(t)>\lambda_iv_i\}} F\Big[
\frac{x_i(t)}{\lambda_i}\Big] \,dt &=&
1-\int_{\{t\in\Delta_i:\,x_i(t)\le\lambda_iv_i\}} F\Big[\frac{x_i(t)}{\lambda_i}\Big] \,dt\\
&\geq& 1-2^{-i-1}\ge \frac34,
\end{eqnarray*}
whence, taking into account the left hand side of
\eqref{main ineq}, we obtain 
$$
\int_{\Delta_i} F[x_i(t)] \,dt\ge C_1^{-1}\lambda_i^p
\int_{\{t\in\Delta_i:\,x_i(t)>\lambda_iv_i\}} F \Big[
\frac{x_i(t)}{\lambda_i} \Big]
\,dt\ge\frac34C_1^{-1}\lambda_i^p,\;\;i\in I_2.
$$
Combining this with \eqref{equality for norm} and \eqref{ineq for I_1}, we get
$$
\sum_{i=1}^n\lambda_i^p=\sum_{i\in I_1}\lambda_i^p +\sum_{i\in
I_2}\lambda_i^p\le \frac14+\frac43 C_1\sum_{i=1}^n\int_{\Delta_i}
F[x_i(t)]\,dt\le 2C_1,
$$
and the first inequality in \eqref{decompForSym7} is proved.

On the other hand, using the right hand side of \eqref{main ineq}
and \eqref{norms of summands}, we infer that
\begin{equation}\label{converse ineq}
\begin{split}
\sum_{i\in I_2}\int_{\{t\in\Delta_i:\,x_i(t)>\lambda_iv_i\}}
F[x_i(t)]\,dt
&\leq
C_1\sum_{i\in I_2}\lambda_i^p\int_{\{t\in\Delta_i:\,x_i(t)>\lambda_iv_i\}}
F \Big[\frac{x_i(t)}{\lambda_i}\Big] \,dt\\
&\leq C_1\sum_{i=1}^n\lambda_i^p.
\end{split}
\end{equation}
At the same time, by \eqref{inegral ineq} and the convexity of $F$, we obtain 
\begin{eqnarray*}
\sum_{i\in I_2}\int_{\{t\in\Delta_i:\,x_i(t)\le\lambda_iv_i\}} F[ x_i(t)] \,dt
&\leq&
\sum_{i\in I_2}\lambda_i \int_{\{t\in\Delta_i:\,x_i(t)\le\lambda_iv_i\}}
F \Big[\frac{x_i(t)}{\lambda_i}\Big] \,dt\\
&\leq& \sum_{i=1}^\infty 2^{-i-1}=\frac12
\end{eqnarray*}
and, by \eqref{norms of summands} and the
definition of $I_1$,
$$
\sum_{i\in I_1}\int_{\Delta_i} F[x_i(t)]\,dt\le \sum_{i\in
I_1}\lambda_i \int_{\Delta_i} F \Big[ \frac{x_i(t)}{\lambda_i}\Big]
\,dt\le\frac14.
$$
Hence, taking into account \eqref{equality for norm}, we get
\begin{eqnarray*}
\sum_{i\in I_2}\int_{\{t\in\Delta_i:\,x_i(t)>\lambda_iv_i\}}
F[x_i(t)]\,dt
&=&
1- \sum_{i\in I_2}\int_{\{t\in\Delta_i:\,x_i(t)\le\lambda_iv_i\}} F[x_i(t)]\,dt \\
&-&
\sum_{i\in I_1}\int_{\Delta_i} F[x_i(t)]\,dt\ge \frac14.
\end{eqnarray*}
From this and \eqref{converse ineq} it follows that
$\sum_{i=1}^n\lambda_i^p\ge 1/(4C_1),$ and so the proof of 
\eqref{decompForSym7} is complete.
\end{proof}

\begin{section}
{\bf Ultrasymmetric Orlicz spaces and $w$--decomposa-\\ bility}
\end{section}

In the previous sections we have examined the problem of the
$K$-monotonicity of weighted couples generated by Lorentz,
Marcinkiewicz and Orlicz spaces. We have seen that the central
role in the question is played by the notion of $w$-decomposibility.
It turns out that studying the last property in
a natural way leads to the so-called ultrasymmetric Orlicz spaces.

Recall that a symmetric space $X$ on $[0, 1]$ is {\it
ultrasymmetric} if $X$ is an interpolation space between the Lorentz
space $\Lambda_{\varphi_X}$ and the Marcinkiewicz space
$M_{\varphi_X}$. These spaces were studied by Pustylnik \cite{Pu03},
who proved that they embrace all possible generalizations of
Lorentz-Zygmund spaces and have a simple analytical description.
Moreover, one could substitute ultrasymmetric spaces into almost all
results concerning classical spaces such as Lorentz-Zygmund spaces,
and so they are very useful in many applications (see, for example,
Pustylnik \cite{Pu05} and \cite{Pu06}).

Pustylnik asked about a description of ultrasymmetric Orlicz spaces
(see \cite{Pu03}, p. 172). In the case of reflexive Orlicz spaces this
problem was solved in \cite{AM08}: such a space is ultrasymmetric if
and only if it coincides (up to equivalence of norms) with a Lorentz
space $\Lambda_{p, \varphi}$ for some $1 < p < \infty$ and some
increasing concave function $\varphi$ on $[0, 1]$.

As it was said above, the class of $w$--decomposable symmetric
spaces is closely related to the class of ultrasymmetric Orlicz
spaces. Our next theorem shows that in the case when a weight
$w$ changes sufficiently fast any $w$--decomposable symmetric space
with non-trivial Boyd indices is an ultrasymmetric Orlicz space.

Again, as above, for a weight $w$ defined on $[0,1]$, let
$M_k:=\{t\in[0,1]:w(t)\in[2^k, 2^{k+1})\}$ $(k\in\mathbb Z)$ and
$(w_k)_{k=1}^{\infty}$ be the non-increasing rearrangement of the
sequence $(m(M_k))_{k=-\infty}^{+\infty}$

\vspace{3mm}
\begin{theorem}\label{ultrasymm. Orlicz}
Let $X$ be a symmetric space on $[0,1]$ with non-trivial Boyd indices and $w$ be a weight 
on $[0, 1]$ satisfying the condition:
\begin{equation}\label{assumption on w}
{\it there ~ are}  ~k_0\in\mathbb N ~{\it and} ~ c_0>0 ~ {\it such ~ that} ~w_k2^k\ge c_0 ~{\it for} ~ k\ge k_0.
\end{equation}
\begin{itemize}
\item[$(a)$] If $X$ is $w$--decomposable, then $X$ is an ultrasymmetric Orlicz space.
\item[$(b)$] If $X$ has the Fatou property and $(X, X(w))$ is a $K$-monotone couple, then $X$ is 
an ultrasymmetric Orlicz space.
\end{itemize}
\end{theorem}

\begin{proof}
(a) Firstly, taking into account the boundedness of the dilation operator and 
Theorem \ref{thm 1}, a symmetric space $X$ is $w$--decomposable
if and only if it is $v$--decomposable, where $v(u)=w(cu)$ for some
$c>0.$ Therefore, we may assume that $c_0=1.$ Denote
$I_k:=[2^{-k},2^{-k+1}),$
$\bar{\chi}_{I_k}:={\chi}_{I_k}/\varphi(2^{-k})$ $(k=1,2,\dots),$
where $\varphi$ is the fundamental function of $X.$
From \eqref{assumption on w} it follows that $m({\rm supp}\bar{\chi}_{I_k})\leq w_k$ for all $k\geq k_0$.
Applying Corollary \ref{cor simplifies thm 3.3} to scalar multiples of $\bar{\chi}_{I_k}$ $(k\geq k_0),$
we get that $(\bar{\chi}_{I_k})_{k=k_0}^\infty$ spans $l_p$ for some $p\in [1,\infty)$ ($p\neq\infty$ because 
the Boyd indices of $X$ are non-trivial). Obviously, replacing $(\bar{\chi}_{I_k})_{k=k_0}^\infty$ 
with $(\bar{\chi}_{I_k})_{k=1}^\infty$ does not change this property, so
for all $a_k\in\mathbb R$ $(k=1,2,\dots)$
$$
\Big\|\sum_{k=1}^\infty
a_k\bar{\chi}_{I_k}\Big\|_X \approx \|(a_k)\|_{l_p}.$$ Then, taking into
account \cite[Proposition~2]{AM08}, we get
$$
X=(L_1,L_\infty)_{l_p((\varphi(2^{-k})2^{-k})_{k=1}^\infty)}^{K}.
$$
By Corollary \ref{cor2}, $\delta_\varphi=\beta_X<1.$ Therefore, 
$\lim_{t \rightarrow \infty} \| \sigma_t \|_{X \rightarrow X}/t = 0$, and we can apply 
\cite[Theorem II.6.6, p.~137]{KPS82} in the case when $A$ is the identity 
operator, to obtain
\begin{eqnarray*}
\|x\|_X \approx \|(\varphi(2^{-k}) \, x^{**}(2^{-k}))_{k=1}^\infty\|_{l_p}
&\approx&
\|(\varphi(2^{-k}) \, x^{*}(2^{-k}))_{k=1}^\infty\|_{l_p}\\
&\approx&
\left( \int_{0}^{1} \left[ x^{*}(t) \varphi(t) \right ]^{p}\,
\frac{dt}{t} \right)^{1/p},
\end{eqnarray*}
and we conclude that
\begin{equation}\label{Lorentz sp.}
X=\Lambda_{p,\varphi}.
\end{equation}
Next, denote
\begin{equation*}
F(u) = \int_{0}^{u} \frac{{\tilde F}(t)}{t} \,dt, ~{\rm where}
~~{\tilde F }(t) =
\begin{cases}
\frac{t}{\varphi^{-1}(1)} & ~{\rm if} ~0 \leq t \leq 1, \\
\frac{1}{\varphi^{-1}(\frac{1}{t})} & ~{\rm if} ~t \geq 1.
\end{cases}
\end{equation*}
Since ${\tilde F}(t)/t$ is increasing on $(0,\infty),$ then $F(u)$
is a convex function and for $u>0$ we have that
$$
{\tilde F}(u/2) \leq \int_{u/2}^{u} \frac{{\tilde F}(t)}{t} \,dt
\leq F(u) \leq  {\tilde F}(u).
$$
Moreover, by Corollary \ref{cor2}, we have that $\gamma_{\varphi}=\alpha_X > 0$,
which implies that ${\tilde F}$ satisfies the
$\Delta_{2}$--condition for all $u>0.$ Therefore, for all $u > 0$
$$
F(u) \geq {\tilde F}(u/2) \geq c {\tilde F}(u),
$$
that is, the functions $F$ and ${\tilde F}$ are equivalent on $(0,
\infty)$.

Now, we recall the following definition due to Kalton \cite{Ka93} (see
also \cite{AM08}, where the notion is used): For an Orlicz function
$F$ and $1 \leq p < \infty$, define the function $\Psi_{F,
p}^{\infty}(u, C)$ for $0 < u \leq 1, C > 1$ to be the supremum
(possibly $\infty$) of all $N$ such that there exist $1 \leq a_{1} <
a_{2} < \ldots < a_{N}, ~\frac{a_{k}}{a_{k-1}} \geq 2$ for $k = 2,
\ldots, N$ such that for all $k$ either $F_{a_{k}}(u) \geq C u^{p}$
or $ u^{p} \geq C F_{a_{k}}(u)$, where $F_{a}(u): =
\frac{F(au)}{F(a)}$ for $a, u > 0$.
\vspace{1mm}

To complete the proof it suffices to verify that for some $C_{0}> 0,
C_{1}> 0$ and $r> 0$ we have that
\begin{equation*}
\Psi_{F, p}^{\infty} (u, C_{0}) \leq C_{1} u^{-r} ~~\mbox{for all}
~~u \in (0, 1].
\end{equation*}
Indeed, once it is done, we can apply Theorem 1 from \cite{AM08} to
conclude that the Orlicz space $L_F$ is ultrasymmetric and that it
coincides with a Lorentz space $\Lambda_{p,\psi}$ generated by some
increasing concave function $\psi.$ Since the fundamental function
of $L_F$ is equivalent to $\varphi$, then $L_F=\Lambda_{p,\varphi},$
and, in view of \eqref{Lorentz sp.}, the proof is complete.

Since the functions
$F$ and $\tilde F$ are equivalent, then, by \cite[Lemma 1]{AM08}, it
is sufficient to prove the inequality for $\tilde F$,
i.e., to prove that for some $C_{0}
> 0, C_{1} > 0$ and $r
> 0$ we have
\begin{equation}\label{Kalton quantity}
\Psi_{{\tilde F}, p}^{\infty} (u, C_{0}) \leq C_{1} u^{-r} ~~\mbox{
for all} ~~u \in (0, 1].
\end{equation}

In view of $w$-decomposability, Corollary \ref{cor2}, Lemma \ref{lemma2} and the inequality $w_k \geq 2^{-k}$, 
there is a constant $C>0$ such that for any $l=1,2,\dots$
\begin{equation*}
\frac{\varphi(lt)}{\varphi(t)} ~ \stackrel{C}{\approx} ~
l^{1/p}\;\;\mbox{ if }\;\; 0<t\le 2^{-l}.
\end{equation*}
Since $0<\alpha_X\le\beta_X<1$ it follows that $0<\gamma_\varphi\le\delta_\varphi<1.$
Therefore, from the definition of $\tilde{F}$ it follows that both $\tilde{F}$ and its complementary
function satisfy the $\Delta_2$-condition. Hence, by Proposition \ref{equiv of functions}
and by the definition of $\tilde{F}$ once more, we obtain that there exists a constant
$C_1>0$ such that, for any $l\in\mathbb N$ and for all $x\ge {\tilde F}^{-1}(2^l)$, we have
$$
\frac{1}{C_1l}\le\frac{\tilde{F}(xl^{-1/p})}{\tilde{F}(x)}\le
\frac{C_1}{l}.
$$
By standard arguments, there are constants $C_2>0$ and $C_3>0$ such that
\begin{equation}\label{relation for the inverse and constant C_0}
C_2^{-1} u^p\le \frac{{\tilde F}(ua)}{{\tilde F}(a)}\le C_2 u^p,
\end{equation}
for all $0<u\le 1$ and any $a$ satisfying $\tilde{F}(a)\ge
C_3 2^{u^{-p}}.$

Suppose that $1 \leq a_{1} < a_{2} < \ldots < a_{N},
~\frac{a_{k}}{a_{k-1}} \geq 2$ for $k = 2, \ldots, N$ such that for
all $k$
$$
\mbox{either}\;\; \frac{{\tilde F}(ua_k)}{{\tilde F}(a_k)}\geq 2C_2
u^{p}\;\;\mbox{or}\;\; \frac{{\tilde F}(ua_k)}{{\tilde F}(a_k)}\leq
\frac{1}{2C_2} u^{p}.
$$
Then, by \eqref{relation for the inverse and
constant C_0}, we have that ${\tilde F}(a_N)\le C_3 2^{u^{-p}},$
which implies ${\tilde F}(a_12^{N-1})\le C_3 2^{u^{-p}}.$ Hence,
$N\le C_4 u^{-p},$ that is, $\Psi_{{\tilde F}, p}^{\infty} (u,
2 C_{2}) \leq C_4u^{-p}$ $(0<u\le 1),$ and \eqref{Kalton quantity} is
proved.

(b) This part follows immediately from (a) and Theorem \ref{Theor Tikhomirov}.
\end{proof}
\vspace{1mm}

Using equality (\ref{Lorentz sp.}) from the proof of Theorem \ref{ultrasymm. Orlicz}, 
we obtain the following corollary.

\begin{corollary}\label{cor5} \label{criterion for L_P}
Let $X$ be a symmetric space on $[0,1]$ and $w$ be a weight on $[0,
1]$ satisfying the condition (\ref {assumption on w}). Assume that
either $X$ is $w$--decomposable or $X$ has the Fatou property and
$(X, X(w))$ is $K$-monotone couple. If $\varphi_X(t) = t^{1/p}$ for
some $1<p<\infty$, then $X=L_p$.
\end{corollary}

\begin{remark} Using Krivine's theorem and the arguments from the beginning of the proof of
Theorem \ref{ultrasymm. Orlicz}, the last assertion can be proved for $p=1$ and $p=\infty$ as well.
\end{remark}

\begin{remark} It is well known that there is a regularly varying at $\infty$
Orlicz function $F$ such that the corresponding Orlicz space $L_F$ is not
ultrasymmetric (see \cite{Ka93}). Thus, Theorems \ref{Cor on non-tri. weights}
and \ref{regularly var and Orlicz} show that condition (\ref{assumption on w})
on the weight $w$ from Theorem \ref{ultrasymm. Orlicz} and
Corollary \ref{criterion for L_P} is essential.
\end{remark}

\begin{remark} Conversely, if $L_F$ is an ultrasymmetric reflexive
Orlicz space on $[0, 1]$, then there is a weight $w$ on $[0, 1]$
such that $L_F$ is $w$-decomposable and, equivalently, the Banach
couple $(L_F, L_F(w))$ is $K$-monotone. In fact, in that case $F$ is
regularly varying at $\infty$ of order $p \in (1, \infty)$ (cf.
\cite{AM08}) and we can apply Theorem \ref{regularly var and
Orlicz}.
\end{remark}

\noindent
{\bf Examples.} Theorem \ref{regularly var and Orlicz} guarantees
that a weighted couple of Orlicz spaces $(L_F, L_F(w))$ on $[0, 1]$
is $K$-monotone for some weight $w$ on $[0, 1]$ if $F$ is equivalent
to an Orlicz function which is regularly varying at $\infty$ of
order $p\in [1,\infty)$. We present some examples of such
Orlicz functions below.

1. The function $F(u) = u^p(1 + |\ln u|)$ for $p \geq (3 + \sqrt{5})/2$ is an Orlicz function on 
$(0, \infty)$ which is regularly varying at $\infty$ of order $p$ (cf. \cite[Example 4]{Ma85}).

2. The function $F(u) = u^p[1 + c \sin (p \ln u)]$ for $0 < c < 1/\sqrt{2}$ and 
$p \geq (1 -\frac{\sqrt{2} c}{\sqrt{1 - 2 c^2}})^{-1}$ is an Orlicz function on $(0, \infty)$ which is not regularly 
varying but it is equivalent to $u^p$ and $\frac{1}{4} u^p \leq F(u) \leq 2 u^p$ for all $u > 0$ 
(cf. \cite[Example 10]{Ma85} and \cite{Ma89}, Example 5 on p. 93 with $c = 1/\sqrt{5}$ and $p \geq 6$).

3. Let an Orlicz function $F$ be equivalent for large $u$ to the function
$$
\tilde F(u) = u^p(\ln u)^{q_1}(\ln \ln u)^{q_2}\dots (\ln\dots \ln u)^{q_n},
$$
where $p\in(1,\infty)$ and $q_1,\dots, q_n$ are arbitrary real numbers.
It is easy to see that $F$ is equivalent to a regularly varying function at
$\infty$ of order $p$ (in fact, the corresponding Orlicz space $L_F$ is even
ultrasymmetric \cite{AM08}).

4. Some more examples of Orlicz functions that are equivalent to some regularly
varying functions at $\infty$ of order $p$ are given by Kalton \cite{Ka93}.


\vspace{3mm}

\noindent
{\footnotesize Sergey V. Astashkin and Konstantin E. Tikhomirov\\
Department of Mathematics and Mechanics, Samara State University\\
Acad. Pavlova 1, 443011 Samara, Russia} \\
{\it E-mail addresses:} {\tt astashkn@ssu.samara.ru}, {\tt ktikhomirov@yandex.ru} \\

\vspace{-3mm}

\noindent
{\footnotesize Lech Maligranda\\
Department of Engineering Sciences and Mathematics\\
Lule\r{a} University of Technology\\
SE-971 87 Lule\r{a}, Sweden}\\
{\it E-mail address:} {\tt lech@sm.luth.se} \\

\end{document}